\documentclass[11pt]{article}
\usepackage{amsmath,amssymb,amsthm,mathrsfs}
\usepackage[dvips]{graphicx}
\usepackage[dvipdfmx]{hyperref}
\usepackage{multirow}

\usepackage[table]{xcolor}

\makeatletter
\@addtoreset{equation}{section}

\makeatother
\makeatletter
\renewcommand\section{\@startsection {section}{1}{\z@}%
             {-3.5ex \@plus -1ex \@minus -.2ex}%
             {2.3ex \@plus.2ex}%
             {\normalfont\large\bfseries}}
\renewcommand\subsection{\@startsection{subsection}{2}{\z@}%
             {-3.25ex\@plus -1ex \@minus -.2ex}%
             {1.5ex \@plus .2ex}%
             {\normalfont\normalsize\bfseries}}
\renewcommand\subsubsection{\@startsection{subsubsection}{3}{\z@}%
             {-3.25ex\@plus -1ex \@minus -.2ex}%
             {1.5ex \@plus .2ex}%
             {\normalfont\normalsize\bfseries}}
\makeatother

\newcommand{\MF}[1]{\mathfrak{#1}}

\newcommand{\Z}{\mathbb{Z}}
\newcommand{\Q}{\mathbb{Q}}
\newcommand{\R}{\mathbb{R}}
\newcommand{\C}{\mathbb{C}}
\renewcommand{\P}{\mathbb{P}}
\newcommand{\rank}{\operatorname{rank}}

\newcommand{\SL}{\mathop{\mathrm{SL}}\nolimits}
\newcommand{\SO}{\mathop{\mathrm{SO}}\nolimits}

\newcommand{\GL}{\mathop{\mathrm{GL}}\nolimits}
\newcommand{\PGL}{\mathop{\mathrm{PGL}}\nolimits}

\newcommand{\TQ}[6]
 {\left(\begin{smallmatrix}
   #1 & #6 & #5 \\
   #6 & #2 & #4 \\
   #5 & #4 & #3
  \end{smallmatrix}\right)}
\newcommand{\II}{\text{II}}

\newcommand{\sym}[1]{{#1_{s}}}
\newcommand{\nsym}[1]{#1}
\newcommand{\disc}[1]{{#1}^\vee  / #1}
\DeclareMathOperator{\Aut}{Aut}
\DeclareMathOperator{\NS}{NS}
\DeclareMathOperator{\End}{End}
\renewcommand{\H}{\operatorname{H}}

\newcommand{\syminv}{\Lambda^{\sym{G}}}
\newcommand{\symco}{\Lambda_{\sym{G}}}
\newcommand\Tstrut{\rule{0pt}{2.6ex}}
\title{On K3 surfaces with maximal symplectic action.}
\title{Extensions of maximal symplectic actions on K3 surfaces}
\author{Simon Brandhorst and Kenji Hashimoto}
\date{\today}
\begin{document}
\theoremstyle{plain}
\newtheorem{thm}{Theorem}[section]
\newtheorem{lem}[thm]{Lemma}
\newtheorem{cor}[thm]{Corollary}
\newtheorem{prop}[thm]{Proposition}
\newtheorem{clm}[thm]{Claim}
\newtheorem{ex}[thm]{Example}
\newtheorem{nota}[thm]{Notation}
\theoremstyle{definition}
\newtheorem{definition}[thm]{Definition}
\newtheorem{rem}[thm]{Remark}
\maketitle

\begin{abstract}\noindent
We classify pairs $(X,G)$ consisting of a complex K3 surface $X$ and a finite group $G~\leq~\Aut(X)$ such that the subgroup $G_s \lneq G$ consisting of symplectic automorphisms is among the $11$ maximal symplectic ones as classified by Mukai.
\end{abstract}
\section{Introduction}
A (complex) \emph{K3 surface} is a compact, complex manifold $X$ of dimension $2$ which is simply connected and admits a no-where degenerate holomorphic symplectic form $\sigma_X \in \H^0(X,\Omega^2_X)$ unique up to scaling.
An automorphism of a K3 surface is called \emph{symplectic} if it leaves the $2$-form invariant and non-symplectic else.
Finite groups of symplectic automorphisms of K3 surfaces were classified by Mukai up to isomorphism of groups.
Namely, a group acts faithfully and symplectically on some complex K3 surface if and only if it admits an embedding into the Mathieu group $M_{24}$ which decomposes the $24$ points into at least $5$ orbits and fixes a point (in particular it is contained in $M_{23}$) \cite{mukai:symplectic,kondo:symplectic}.
This leads to a list of $11$ maximal subgroups (with $5$ orbits) among the subgroups of $M_{24}$ meeting these conditions.
A finer classification, namely up to equivariant deformation, was obtained in \cite{hashimoto:symplectic}. There are $14$ maximal finite symplectic group actions (see Table \ref{tbl:symplectic}).

However not every automorphism of a K3 surface is symplectic.
Let $X$ be a K3 surface and $G \leq \Aut(X)$ a group of automorphisms. We remark that $G$ is finite if and only if there is an ample class on $X$ invariant under $G$. Denote by $\sym{G}$ the normal subgroup consisting of symplectic automorphisms. Let $G$ be finite. Then we have a natural exact sequence
\[1 \rightarrow \sym{G} \rightarrow \nsym{G} \overset{\rho}{\rightarrow} \mu_n \rightarrow 1\]
where $n \in \{k \in \mathbb{N} \mid \varphi(k) \leq 20 \}$
 and $\varphi$ is the Euler totient function.
The homomorphism $\rho$ is defined by $g^* \sigma_X=\rho(g)\cdot\sigma_X$.
In the present paper, we classify finite groups $\nsym{G}$ of automorphisms of K3 surfaces, under the condition that $\sym{G}$ is among the $11$ maximal groups and $G_s \lneq G$. As it turns out, this forces the underlying K3 surface $X$ to have maximal Picard number $20$, i.e.\ it is a singular K3 surface. In particular it has infinite automorphism group. Moreover, those K3 surfaces (with $G$) are rigid (i.e.\ not deformable). % and so is the pair $(X,\nsym{G})$.
Let $(X,\nsym{G})$ and $(X',\nsym{G}')$ be two pairs of K3 surfaces with a group of automorphisms. They are called isomorphic if there is an isomorphism $f \colon X \rightarrow X'$ with $f G f^{-1} = G'$.
\begin{thm}\label{thm:classification}
 Let $X$ be a K3 surface and $\nsym{G}\leq \Aut(X)$ a maximal
 finite group of automorphisms such that the symplectic part
 $\sym{G}$ is isomorphic to one of the $11$ maximal groups
 and $\sym{G} \lneq \nsym{G}$. Then the pair $(X,\nsym{G})$ is
 isomorphic to one of the $42$ pairs listed in Section
 \ref{sec:classification}.
\end{thm}
 The representations of the groups $G$ on the K3 lattices
 $\Lambda\cong \H^2(X,\Z)$ are given in the ancillary
 file to \cite{brandhorst-hashimoto:arxiv} on arXiv.

The proof goes via a classification, up to conjugacy,
of suitable finite subgroups of the orthogonal group of the K3 lattice.
Then the strong Torelli type theorem \cite{shafarevic:torelli,burns-rapoport:1975}
and the surjectivity of the period map \cite{todorov:1980}
abstractly provide the existence and uniqueness of the
pairs $(X,G)$. Note that a K3 surface admitting a non-symplectic
automorphism of finite order must be projective \cite[Thm. 3.1]{nikulin:auto}. Thus,
at least in principle, it is possible to find projective models
of the K3 surfaces and the automorphisms.
For $25$ out of the $42$ pairs $(X,G)$ we list explicit equations in Section \ref{sec:classification}.
Using the Torelli-type theorem Kondo proved in \cite{kondo:maximal} that the maximal order
of a finite group of automorphisms of a K3 surface
equals $3840$. Section \ref{sect:maximal} is devoted to deriving its equations for the first time.

For the full story of symplectic groups of automorphisms we recommend the excellent survey \cite{kondo:symplectic_survey}.
Non-symplectic automorphisms of prime order are treated in \cite{artebani_sarti_taki:non-symplectic}.
In \cite{frantzen:automorphisms}, a similar
classification with different methods is carried out,
albeit under the restrictive condition that
$\nsym{G} = \sym{G} \times \mu_2$ and the group $\mu_2$ has fixed points.
Note that the author misses cases \textbf{70d} and \textbf{76a} (see Section \ref{sec:classification}).

\begin{rem}
 Let $(X,\nsym{G})$ be as in Theorem \ref{thm:classification}.
 It turns out that the non-symplectic part $\nsym{G}/\sym{G}\cong \mu_n$
 is of even order and the pair $(X,G)$ is determined up to isomorphism already
 by $\sym{G}$ and any involution in $\nsym{G}/\sym{G}$. See section \ref{sect:extensions} for details.
\end{rem}

\noindent \textbf{Open problems}

\noindent
We close this section with some interesting problems concerning  groups of automorphisms of K3 surfaces.
\begin{enumerate}
\itemsep0pt
 \item Find the remaining $17$ missing equations among the $42$ K3 surfaces and their automorphisms.
 \item Give generators of the full automorphism group of the corresponding K3 surfaces. Since a Conway chamber in the nef cone of this surface has large symmetry, chances are that one can find a nice generating set for the automorphism group.
 \item Find a projective model of the K3 surface with a linear action by $M_{22}$ in characteristic $11$. Its existence is proven by Kondo \cite{kondo:m22} using  the crystalline Torelli type theorem.
 \item Use the present classification to study finite groups of automorphisms of Enriques surfaces beyond the semi-symplectic case \cite{mukai-ohashi:m12}.
\end{enumerate}
Finding equations for the surface is often much easier than for the automorphisms. Should you find equations or relevant publications on one of the surfaces treated here, please notify the first author. We will update the arXiv version of this paper with your findings.\\

\noindent
\textbf{Acknowledgments.}
We would like to thank the organizers of the conference
Moonshine and K3 surfaces in Nagoya in 2016 where the idea for this work was born. The first author would like to thank the University of Tokyo and Keiji Oguiso for their hospitality.
Thanks to Matthias Sch\"utt for encouragement and discussions. We warmly thank C\'edric Bonnaf\'e, Noam Elkies, Hisanori Ohashi and Alessandra Sarti for sharing explicit models of symmetric K3 surfaces with us.
We also thank the anonymous referee for carefully reading our manuscript and suggesting many improvements.
S.~B. is supported by  SFB-TRR 195 ”Symbolic Tools in Mathematics and their Application” of the German Research Foundation (DFG).
 K.~H. was partially supported by Grants-in-Aid for Scientific Research (17K14156).

\section{Lattices}
In this section we recall the basics on integral lattices (equivalently quadratic forms) and fix notation. The results are found in \cite{nikulin:quadratic_forms,conway-sloane:sphere_packings}.\\

A \emph{lattice} consists of a finitely generated free $\Z$-module $L$ and a non-degenerate integer valued symmetric bilinear form
\[\langle \cdot \, , \cdot \rangle \colon L \times L \rightarrow \Z.\]
Given a basis $(b_1,\dots, b_n)$ of $L$, we obtain the \emph{Gram matrix} $Q= (\langle b_i,b_j \rangle)_{1 \leq i,j \leq n }$. The determinant $\det Q$ is independent of the choice of basis and called the \emph{determinant} of the lattice $L$; it is denoted by $\det L$. We display lattices in terms of their Gram matrices. The \emph{signature} of a lattice is the signature of its Gram matrix. We denote it by $(s_+,s_-)$ where $s_+$ (respectively $s_-$) is the number of positive (respectively negative) eigenvalues.
We define the \emph{dual lattice} $L^\vee$ of $L$ by $L^\vee=\{x \in L \otimes \Q | \langle x, L \rangle \subseteq \Z \} \cong Hom(L,\Z)$. The \emph{discriminant group} $L^\vee/L$ is a finite abelian group of cardinality $|\det L|$.
We call a lattice \emph{unimodular} if $L=L^\vee$, and we call it \emph{even} if $\langle x , x \rangle$ is even for all $x \in L$.
The discriminant group of an even lattice carries the \emph{discriminant form}
\[q_L \colon L^\vee / L \rightarrow \Q / 2\Z, \quad \bar x \mapsto \langle x,x\rangle + 2\Z.\]
An \emph{isometry} of lattices is a linear map compatible with the bilinear forms. The \emph{orthogonal group} $O(L)$ is the group of isometries of $L$ and the special orthogonal group $\SO(L)$ consists of the isometries of determinant $1$. Discriminant forms are useful to describe embeddings of lattices and extensions of isometries.
A sublattice $L \subseteq M$ is called \emph{primitive}, if $L = (L \otimes \Q) \cap M$. By definition, the orthogonal complement $S^\perp\subseteq M$ of a (not necessarily primitive) sublattice $S$ is a primitive sublattice. For $L_1$ primitive and $L_2 = L_1^\perp$ we call $L_1 \oplus L_2 \subseteq M$ a \emph{primitive extension}.
Now, suppose that $M$ is even, unimodular, then
\[H_M = M/(L_1 \oplus L_2) \subseteq (\disc{L_1}) \oplus (\disc{L_2})\]
is the graph of a so called \emph{glue map} $\phi_M\colon \disc{L_1} \rightarrow \disc{L_2}$, that is, any element in $H_M$ is of the form $x \oplus \phi_M(x)$ for $x \in \disc{L_1}$. This isomorphism is an anti-isometry, namely, it satisfies $q_{L_2} \circ \phi_M = -q_{L_1}$.
Conversely given such an anti-isometry $\phi$, its graph $H_\phi$ defines a primitive extension $L_1\oplus L_2\subseteq M_\phi$ with $M_\phi$ even, unimodular.

Given an isometry $f \in O(L_1)$, it induces an isometry $\bar f \in O(\disc{L_1})$ of the discriminant group.
Let $g\in O(L_2)$ be an isometry on the orthogonal complement. Then $f\oplus g \in O(L_1 \oplus L_2)$ extends to $M$ if and only if $(\bar f \oplus \bar g)(H_M)=H_M$, or equivalently, $\phi_M \circ \bar f = \bar g \circ \phi_M$.

\begin{lem}\label{lem:extend}
 Let $L \subseteq M$ be a primitive sublattice of an even unimodular lattice $M$. Set $O(M,L)=\{f \in O(M) | f(L) = L\}$ and $K = L^\perp$. If the natural map $O(K) \rightarrow O(\disc{K})$ is surjective, then the restriction map $O(M,L)\rightarrow O(L)$ is surjective. In other words: any isometry of $L$ can be extended to an isometry of $M$.
\end{lem}
\begin{proof}
 Denote the glue map by $\phi=\phi_M$, and let $g \in O(K)$ be a preimage of $\phi \circ \bar f \circ \phi^{-1}$. Then $\phi \circ \bar f = \bar g \circ \phi$. Hence, $f \oplus g$ extends to $M$.
\end{proof}

Let $L$ be a lattice and $G \leq O(L)$. We define the \emph{invariant} and \emph{coinvariant lattices} respectively by
\[L^G = \{x \in L  \mid\forall g \in G\colon g(x)=x \} \quad\text{and}\quad L_G = \left(L^G\right)^\perp.\]
Then, by definition, $L^G \oplus L_G \subseteq L$
is a primitive extension.
Two lattices are said to be in the same genus, if they become isometric after tensoring with the $p$-adics $\Z_p$ for all primes $p$ and the reals $\R$. A genus is denoted in terms of the Conway-Sloane symbols \cite[Chap. 15]{conway-sloane:sphere_packings}.
For instance the genus of even unimodular lattices of signature $(3,19)$ is denoted by $\II_{3,19}$. In fact all lattices in this genus are isometric.

\section{K3 surfaces and the Torelli type theorem}
In this section we recall standard facts about complex K3 surfaces.
All results can be found in the textbooks \cite{BHPV:compact_complex_surfaces, huybrechts:k3-book}.

Let $X$ be a K3 surface. Its second integral cohomology group $\H^2(X,\mathbb{Z})$
together with the cup product is an even unimodular lattice of signature $(3,19)$.
It comes equipped with an integral weight $2$ Hodge structure.
Such a Hodge structure is given by its Hodge decomposition
\[\H^2(X,\Z) \otimes \C = \H^2(X,\C)= \H^{2,0}(X)\oplus \H^{1,1}(X) \oplus \H^{0,2}(X)\]
with $\H^{i,j}(X) = \overline{\H^{j,i}}(X)$ and natural isomorphisms $\H^{i,j} \cong \H^j(X,\Omega_X^i)$. The corresponding Hodge numbers are $h^{2,0}=h^{0,2}=1$ and $h^{1,1}=20$.
We can recover the entire Hodge structure from $\H^{2,0}(X)$ via $\H^{0,2}(X)=\overline{\H^{2,0}(X)}$ and $\H^{1,1}(X) = \left(\H^{2,0}(X) \oplus \H^{0,2}(X)\right)^\perp$.
The \emph{transcendental lattice} of a K3 surface is defined as
the smallest primitive sublattice $T_X$ of $\H^2(X,\Z)$ such that $T_X\otimes \C$ contains the period $\H^{2,0}(X)=\C \sigma_X$.
By the Lefschetz theorem on $(1,1)$-classes, the N\'eron-Severi lattice $\NS_X$ of a K3 surface is given by
$\H^{1,1}(X) \cap \H^2(X,\Z)$. Note that $\NS_X$ and $T_X$ can be degenerate \cite[(3.5)]{nikulin:auto}. But if $X$ is projective, then they are (non-degenerate) lattices of signatures $(1,\rho-1)$ and $(2,20-\rho)$ respectively, and we have $\NS_X = T_X^\perp$.

As a next step we want to compare Hodge structures of different K3 surfaces. For this we fix a reference frame, namely a lattice $\Lambda \in \II_{3,19}$.
\begin{definition}
 A \emph{marked} K3 surface is a pair $(X,\eta)$ consisting of a complex K3 surface $X$ and an isometry $\eta \colon \H^2(X,\Z) \rightarrow \Lambda$. We call $\eta$ a \emph{marking}.
\end{definition}

We associate a marked K3 surface $(X,\eta)$ with its \emph{period}
\[\eta_\C\left(\H^{2,0}(X)\right) \in \mathcal{P}_\Lambda:=\{ \C \sigma \in \mathbb{P}(\Lambda \otimes \C) \mid \langle \sigma,\bar \sigma \rangle > 0, \langle \sigma,\sigma \rangle =0 \}.\]
Here we extend the bilinear form on $\Lambda$ linearly to that on $\Lambda \otimes \C$.
We call $\mathcal{P}_\Lambda$ the \emph{period domain}.
As it turns out, the concept of marking works well in families. This allows one to define the moduli space $\mathcal{M}_\Lambda$ of marked K3 surfaces and a period map
\[ \mathcal{M}_\Lambda \rightarrow \mathcal{P}_\Lambda, \quad (X,\eta) \mapsto \eta_\C\left(\H^{2,0}(X)\right).\]
The period map is in fact holomorphic, and it turns out to be surjective as well (the surjectivity of the period map for K3 surfaces \cite{todorov:1980}). The moduli space $\mathcal{M}_\Lambda$ is not very well behaved. For example it is not Hausdorff. This can be healed by taking into account the K\"ahler (resp.\ ample) cone.

The \emph{positive cone} $\mu_X$ is the connected component of the set
\[\{x \in \H^{1,1}(X,\R) \mid \langle x, x \rangle > 0\}\]
which contains a K\"ahler class.
Set $\Delta_X = \{x \in \NS_X \mid \langle x,x\rangle = -2\}$.
An element in $\Delta_X$ is called a root.
%Let $\delta \in \NS_X$ be of square $-2$. 
For $\delta \in \Delta_X$, either $\delta$ or $-\delta$ is an effective class by the Riemann--Roch theorem. In fact the effective cone is generated by the effective classes in $\Delta_X$ and the divisor classes in the closure of the positive cone (i.e.\ $\NS_X \cap \overline{\mu_X}$).
The connected components of the set $\mu_X \setminus \bigcup_{\delta \in \Delta_X} \delta^\perp$ are called the \emph{chambers}. The hyperplanes $\delta^\perp$ for $\delta \in \Delta_X$ are called the \emph{walls}.
One of the chambers is the K\"ahler cone.
For a root $\delta \in \Delta_X$, the reflection with respect to the wall $\delta^\perp$ is given by $r_\delta(x) = x + \langle x , \delta \rangle \delta$. The \emph{Weyl group} is the subgroup of $O(\H^2(X,\Z))$ generated by the reflections $r_\delta$ for $\delta \in \Delta_X$. The action of the Weyl group on the chambers is simply transitive. So by composing the marking with an element of the Weyl group, we can ensure that any given chamber in the positive cone of $\Lambda$ corresponds to the K\"ahler cone.
\begin{definition}
 Let $X,X'$ be K3 surfaces. An isometry $\phi\colon \H^2(X,\Z) \rightarrow \H^2(X',\Z)$ is called a \emph{Hodge isometry} if $\phi_\C(\H^{i,j}(X))\subseteq \H^{i,j}(X')$ for all $i,j$.
 It is called \emph{effective}, if it maps effective (resp. K\"ahler, resp. ample) classes on $X$ to effective (resp.\ K\"ahler, resp.\ ample) classes on $X'$.
\end{definition}
The following Torelli type theorem for K3 surfaces is the key tool for our classification of automorphisms.
\begin{thm}[{\cite{shafarevic:torelli,burns-rapoport:1975}}]
 Let $X$ and $X'$ be complex K3 surfaces. Let
 \[\phi \colon \H^2(X,\Z) \rightarrow \H^2(X',\Z)\]
 be an effective Hodge isometry. Then there is a unique isomorphism
 $f \colon X' \rightarrow X$ with $f^* = \phi$.
\end{thm}
We thus obtain a Hodge theoretic characterization of the automorphism group of a K3 surface.
\begin{cor}
 Let $X$ be a complex K3 surface. Then the image of the natural homomorphism
 \[\Aut(X) \rightarrow O(\H^2(X,\Z))\]
 consists of the isometries preserving the period and the K\"ahler cone.
\end{cor}

\section{Symplectic automorphisms}
In this section we review known facts on symplectic automorphisms needed later on.

Let $X$ be a complex K3 surface.
We obtain an exact sequence
\[1 \rightarrow \sym{\Aut(X)} \rightarrow \Aut(X) \xrightarrow{\rho} \GL(\C \sigma_X).\]
(Recall that we have $\C \sigma_X = \H^0(X,\Omega_X^2)$.)
The elements of the kernel $\sym{\Aut(X)}$ of $\rho$ are the symplectic automorphisms. An automorphism which is not symplectic is called non-symplectic. If $G\leq \Aut(X)$ is a group of automorphisms, we denote by $\sym{G}$ the kernel of $\rho|_G$ and call it the symplectic part of $\nsym{G}$. In order to keep the notation light, we identify $\nsym{G}$ and its isomorphic image in $O(\H^2(X,\Z))$.

Recall that if $L$ is a lattice and $G\leq O(L)$, then $L^G$ is the invariant and $L_G=\left(L^G\right)^\perp$ the coinvariant lattice.
For the sake of completeness we give a proof of the following essential lemma.
\begin{lem}[cf.\ {\cite{nikulin:auto}}]\label{lem:symplectic}
Let $\sym{G}\leq \sym{\Aut(X)}$ be a finite group of symplectic automorphisms of some K3 surface $X$. Then
\begin{enumerate}
 \item[(1)] $T_X \subseteq \H^2(X,\Z)^\sym{G}$ and $\H^2(X,\Z)_\sym{G} \subseteq \NS_X$;
 \item[(2)] $\H^2(X,\Z)^\sym{G}$ is of signature $(3,k)$ for some $k\leq 19$;
 \item[(3)] $\H^2(X,\Z)_\sym{G}$ is negative definite;
 \item[(4)] $\H^2(X,\Z)_\sym{G}$ contains no vectors of square $-2$;
 \item[(5)] if $\sym{G}$ is maximal (that is, $G_s$ is isomorphic to one of the 11 maximal finite groups of symplectic automorphisms), then $\sym{G}\cong\ker \left(O(H) \rightarrow O(\disc{H})\right)$ where $H=\H^2(X,\Z)_\sym{G}$.
\end{enumerate}
\end{lem}
\begin{proof}
(1) The elements of $\sym{G}$ are all symplectic, i.e.\ they fix the $2$-form $\sigma_X$. Thus $\C \sigma_X \subseteq \H^2(X,\Z)^\sym{G} \otimes \C$. By minimality of the transcendental lattice and primitivity of the invariant lattice, we get $T_X \subseteq \H^2(X,\Z)^\sym{G}$. Taking orthogonal complements yields the second inclusion.\newline
(2) Let $\kappa'$ be a K\"ahler class. Since automorphisms preserve the K\"ahler cone, the class $\kappa = \sum_{g \in G} g^*\kappa'$ is a $\sym{G}$-invariant K\"ahler class. Thus $\kappa, (\sigma_X + \bar \sigma_X)/2$ and $(\sigma_X - \bar \sigma_X)/(2i)$ span a positive definite subspace of dimension $3$ of $\H^2(X,\R)^\sym{G}$.\newline
(3) Recall that $H^2(X,\Z)_{\sym{G}}=\left(H^2(X,\Z)^{\sym{G}}\right)^\perp$, and $\H^2(X,\Z)$ has signature $(3,19)$. Now, use (2).
\newline
(4) As before we take a $\sym{G}$-invariant K\"ahler class $\kappa$.
If $r \in \NS_X$ is of square $-2$, then either $r$ or $-r$ is effective by the Riemann--Roch theorem. Thus $\langle \kappa,r \rangle\neq 0$.
Since $\H^2(X,\Z)_\sym{G}$ is orthogonal to $\kappa$, it cannot contain $r$.\newline
(5) Let $g$ be an element in the kernel.
Since $g$ acts trivially on $\disc{H}$, it can be extended to an
isometry $\tilde{g}$ on $H^2(X,\Z)$ such that
$\tilde{g}|_{H^\perp} = \mathop{id}_{H^\perp}$.
As $H^\perp\otimes \C$ contains $\sigma_X$ and a K\"ahler
class, $\tilde{g}$ is in fact an effective Hodge isometry.
The strong Torelli type theorem implies that it is induced
by a symplectic automorphism. Since the coinvariant lattice $H$
is negative definite (by (3)), $O(H)$ is finite. In particular,
the group $\tilde{G}$ generated by $\sym{G}$ and $g$ is a
finite group. By the maximality of $\sym{G}$, $\sym{G}$ must contain $g$.
\end{proof}

% For $\sym{G}$ among the $11$ maximal groups the invariant lattice
% is given in Table \ref{tbl:symplectic}.
% The key observation we take from
% Lemma \ref{lem:symplectic} is that the invariant lattice is definite (of rank $3$)
% and so is the coinvariant lattice. Hence, the group $\sym{G}$ sits
% inside the direct product of the two finite
% groups $O\left(H^2(X,\Z)^\sym{G}\right) \times O\left(H^2(X,\Z)_\sym{G}\right)$. They can be computed explicitly with the Plesken Souvignier algorithm \cite{plesken-souvignier:orthogonal}) as implemented for instance in PARI \cite{pari}.

\begin{thm}[\cite{hashimoto:symplectic}]
Let $\sym{G}$ be a finite group of symplectic automorphisms of a $\Lambda$-marked K3 surface. Identify $\sym{G}$ with its image in $O(\Lambda)$. Then the conjugacy class of $\sym{G}$ is determined by the isometry class of the invariant lattice $\Lambda^\sym{G}$.
For maximal $G_s$, the invariant lattices can be found in Table \ref{tbl:symplectic}, and the coinvariant lattice $\Lambda_{G_s}$ is uniquely determined up to isomorphism by the abstract group structure of $G_s$.
\end{thm}

For maximal $\sym{G}$, we have $\rank \syminv=3$ and $\rank \symco=19$ 
\cite{mukai:symplectic}.
The key observation we take from
Lemma \ref{lem:symplectic}, is that the invariant lattice is definite
and so is the coinvariant lattice. Hence, the direct product $O\left(H^2(X,\Z)^\sym{G}\right) \times O\left(H^2(X,\Z)_\sym{G}\right)$ is a finite group.
It can be computed explicitly with the Plesken-Souvignier algorithm \cite{plesken-souvignier:orthogonal} as implemented for instance in PARI \cite{pari}.
As it turns out the groups $G \leq \Aut(X)$ we aim to classify
are subgroups of this product.

\begin{table}
\caption{Maximal finite symplectic groups of automorphisms}
\label{tbl:symplectic}
\begin{equation*}
\begin{array}{c|c|c|c|c|c|c|c|c} \hline \Tstrut
\text{No.} & \sym{G} & \#\sym{G} & \det \Lambda^\sym{G}& \text{genus of } \Lambda^\sym{G}& \Lambda^\sym{G} & \SO(\syminv) & \#O(\symco) & \#O\!\left(q_{\symco}\right)\\ \hline
\rowcolor{blue!0} 54 & T_{48} & 48 & 384 & 2^{+1}_1,8^{-2}_\II,3^{+1} & \TQ{2}{16}{16}{8}{0}{0} & D_{6} & 9216 & 192\\
\rowcolor{blue!8} 62 & N_{72} & 72 & 324 & 4^{+1}_7,3^{+2},9^{+1} & \TQ{6}{6}{12}{3}{3}{0} & D_4 & 20736 & 288\\
\rowcolor{blue!0} 63 & M_9 & 72 & 216 & 2^{-3}_1,3^{+1},9^{+1} & \TQ{2}{12}{12}{6}{0}{0} & D_{6} & 5184& 72\\
\rowcolor{blue!8} &&&&& \TQ{4}{4}{20}{0}{0}{1} & D_2 && \\
\rowcolor{blue!8} \multirow{-2}{*}{$70$} & \multirow{-2}{*}{$\MF{S}_5$} & \multirow{-2}{*}{$120$} & \multirow{-2}{*}{$300$} & \multirow{-2}{*}{$4^{-1}_5,3^{-1},5^{-2}$} &  \TQ{4}{6}{16}{1}{2}{2} & D_2 & \multirow{-2}{*}{$5760$} & \multirow{-2}{*}{$48$}\\
\rowcolor{blue!0} &&&&& \TQ{2}{4}{28}{0}{0}{1} & D_2 & &\\
\rowcolor{blue!0} \multirow{-2}{*}{$74$} & \multirow{-2}{*}{$L_2(7)$} & \multirow{-2}{*}{$168$} & \multirow{-2}{*}{$196$} & \multirow{-2}{*}{$4^{+1}_7,7^{+2}$} & \TQ{4}{8}{8}{1}{2}{2} & D_4 &\multirow{-2}{*}{$5376$} & \multirow{-2}{*}{$32$}\\
\rowcolor{blue!8} 76 & H_{192} & 192 & 384 & 4^{-2}_4,8^{+1}_1,3^{+1} & \TQ{4}{8}{12}{0}{0}{0} & D_4 & 24576 & 128\\
\rowcolor{blue!0} 77 & T_{192} & 192  & 192 & 4^{-3}_1,3^{-1} & \TQ{4}{8}{8}{4}{0}{0} & D_{6} & 36864 & 192\\
\rowcolor{blue!8} 78 & \MF{A}_{4,4} & 288 & 288 & 2^{+2}_\II,8^{+1}_7,3^{+2} & \TQ{8}{8}{8}{2}{4}{4} & D_4 & 36864 & 128\\
\rowcolor{blue!0} &&&&& \TQ{2}{8}{12}{0}{0}{1} & D_2 & & \\
\rowcolor{blue!0} \multirow{-2}{*}{$79$} & \multirow{-2}{*}{$\MF{A}_6$} & \multirow{-2}{*}{$360$} & \multirow{-2}{*}{$180$} & \multirow{-2}{*}{$4^{-1}_3,3^{+2},5^{+1}$} &\TQ{6}{6}{8}{3}{3}{0} & D_4 & \multirow{-2}{*}{$11520$} & \multirow{-2}{*}{$32$}\\
\rowcolor{blue!8} 80 & F_{384} & 384 & 256 & 4^{+1}_1,8^{+2}_2 & \TQ{4}{8}{8}{0}{0}{0} & D_4 & 49152 & 128\\
\rowcolor{blue!0} 81 & M_{20} & 960 & 160 & 2^{-2}_\II,8^{+1}_7,5^{-1} & \TQ{4}{4}{12}{2}{2}{0} & D_4 & 92160 & 96\\ \hline
\end{array}
\end{equation*}
No.\ denotes the number of the group $\sym{G}\leq O(\Lambda)$ as given in \cite{hashimoto:symplectic}. It is isomorphic to the corresponding group in the column $\sym{G}$. See \cite{mukai:symplectic} for the notation. The entry genus is given in Conway and Sloane's \cite{conway-sloane:sphere_packings} notation. The dihedral group of order $2k$ is denote by $D_k$ .
\end{table}

\section{Non-symplectic extensions}\label{sect:extensions}
In this section we prove the classification.
The invariant lattices of the symplectic actions play a major role.
For a start we observe that the cyclic group $\nsym{G}/\sym{G}$ acts on the invariant lattice.
Indeed for $g \in \nsym{G}$ and $x\in \H^2(X,\Z)^{\sym{G}}$, $g \sym{G} (x) = g(x)$ is well defined and lies in the invariant lattice since $\sym{G}$ is normal in $\nsym{G}$ and fixes $x$.
This yields a homomorphism
\[\nsym{G}/\sym{G} \rightarrow O(\H^2(X,\Z)^{\sym{G}})\]
of groups which turns out to be injective.

\begin{lem}\label{lem:conditions}
Let $X$ be a K3 surface and $G \leq \Aut(X)$ a finite group of automorphisms such that the subgroup $\sym{G} \leq \nsym{G}$ of symplectic automorphisms is among the $11$ maximal ones.
% Then the image of $\nsym{G}/\sym{G} \rightarrow O(\H^2(X,\Z)^{\sym{G}})$ is a cyclic subgroup of $\SO(\H^2(X,\Z)^{\sym{G}})$.
Then the homomorphism $\nsym{G}/\sym{G} \rightarrow O(\H^2(X,\Z)^{\sym{G}})$ is injective and its image is a cyclic subgroup of $\SO(\H^2(X,\Z)^{\sym{G}})$. In particular its order is $n \in \{ 1,2,3,4,6\}$.
\end{lem}
\begin{proof}
By our assumption $\sym{G}$ is maximal. Thus, by Table \ref{tbl:symplectic},  $\H^2(X,\mathbb{Z})^\sym{G}$ is of rank $3$.
% and $\chi=\det(xI - g|\H^2(X,\mathbb{Z})^\sym{G})$ of degree $3$. 
Hence a basis of $\H^2(X,\mathbb{R})^\sym{G}$ is given by a $G$-invariant K\"ahler class $\kappa$, $(\sigma_X +\bar \sigma_X)/2$ and $(\sigma_X -\bar \sigma_X)/(2i)$ (see the proof of Lemma \ref{lem:symplectic}).
Since $\nsym{G}/\sym{G}$ acts on $\H^{2,0}(X)=\C \sigma_X$ faithfully (by the definition of $\sym{G}$), the injectivity in the statement of the lemma follows. By the same reason, $\nsym{G}/\sym{G}$ is cyclic.

Let $g\sym{G}$ be a generator of $\nsym{G}/\sym{G}$.
Then $(x-1)$ divides $\chi(x)=\det(x \mbox{Id} - g|_{\H^2(X,\Z)^\sym{G}})$.
Since $\H^2(X,\mathbb{R})^\sym{G}$ is actually defined
over $\Q$ and $g$ is of finite order, $\chi(x)$ is
a product of cyclotomic polynomials. Note that
the eigenvectors $\sigma_X$ and $\bar \sigma_X$ have
complex conjugate eigenvalues.
Hence $\chi(x)\neq (x+1)(x-1)^2$. This leaves us with
%By Lemma \ref{lem:symplectic}, the transcendental lattice $T_X$ is contained in $\H^2(X,\mathbb{Z})^\sym{G}$.
% The minimal rank for $T_X$ for K3 surfaces is $2$ and $\kappa \subseteq H^{1,1}(X,\RR)$. Thus
% $T_X \oplus \RR \kappa \subseteq \H^2(X,\mathbb{R})^\sym{G}$ is a subspace of full dimension $3$.
%Since $T_X \otimes \C$ is generated by $\sigma_X$ and $\bar \sigma_X$, the characteristic polynomial of $g|T_X$ is
$\chi(x)/(x-1)$ to be one of $(x \pm 1)^2$ or $\Phi_n$ for $n\in \{3,4,6\}$. We conclude by computing the determinant from the characteristic polynomial.
 \end{proof}
Recall that via a marking we may identify $\H^2(X,\Z)$ and $\Lambda$.
\begin{rem}
A choice of basis turns the groups $\SO(\syminv)$ into subgroups of $\SL(\Z^3)$.
Finite subgroups of $\SO(3)$ are an essential building block for crystallographic groups.
It is known that they are isomorphic to a subgroup of a dihedral group, or one of $(\Z/2\Z)^2 \rtimes C_3 \leq (\Z/2\Z)^2 \rtimes S_3$ (cf. \cite[Table I]{auslander-cook:crystallographic}). Up to conjugation there
are exactly $3$ subgroups of the second type, i.e. $3$ integral representations of $(\Z/2\Z)^2 \rtimes C_3$. Since the three representations are irreducible, there is, up to homothety, a unique invariant quadratic form for each. Their gram matrices are given by
\[
\begin{pmatrix}
 1&0&0\\
 0&1&0\\
 0&0&1
\end{pmatrix},
\begin{pmatrix}
 2  &1& 2\\
 1  &2 &2\\
2& 2&  4
\end{pmatrix},
\begin{pmatrix}
 3 &2& 2\\
2 & 4 & 0\\
2 & 0&  4
\end{pmatrix}\]
and are of determinant $1$, $2^2$ and $2^4$. Obviously none of the invariant lattices in Table \ref{tbl:symplectic} is homothetic to one of these three. Thus $\SO(\syminv)$ must be a
(subgroup of) a dihedral group.
\end{rem}

Part 2 of the next lemma will be used later in the proof of Proposition \ref{prop:k3fromgroup}.
\begin{lem}\label{lem:orientation-reversal}
Let $H$ be one of the $14$ symplectic fixed lattices and $g \in \SO(H)$ an involution. Then
\begin{enumerate}
\item $\SO(H)$ is isomorphic to a dihedral group $D_k$ of order $2k$ with $k \in \{2,4,6\}$;
\item there is another involution $f\in \SO(H)$ in the centralizer of $g$ with $\det f|H_g=-1$ where $H_g$ denotes the coinvariant lattice of the group generated by $g$.
\end{enumerate}
\end{lem}
\begin{proof}
 The proof of 1.\ is by a direct computation of $\SO(H)$ for each case using a computer.
 For the reader's entertainment we calculate No.\ 63 by hand.
 The fixed lattice is given as the orthogonal direct sum $(2) \oplus A_2(6)$ where
 $A_2(6)$ is a rescaled hexagonal lattice. The orthogonal group of the hexagonal lattice is
 that of the hexagon, i.e. it is the dihedral group $D_6$. Since the decomposition of a lattice into irreducible lattices is unique up to ordering, the orthogonal group preserves this decomposition and is isomorphic to $\{\pm 1\} \times D_6$. It contains the special orthogonal group with index $2$. The restriction of the $\SO$ action to the hexagonal plane is by $D_6$. Thus it is in fact isomorphic to $D_6$.

 For 2. fix $k\in \{2,4,6\}$ and let $g \in D_k$ be an involution.
 Then it is not hard to check, that there exists an involution $f$ different from $g$ and commuting with $g$. Now view $f$ and $g$ as elements of $\SO(H)$. Then the characteristic polynomials of both are equal to $(x-1)(x+1)^2$. If $\det f|H_g = 1$, then $f|H_g = -\text{id}$. This implies $f=g$, which we excluded.
\end{proof}
Recall the exact sequence
\[1 \rightarrow \sym{G} \rightarrow \nsym{G} \rightarrow \mu_n \rightarrow 1.\]
From the proof of Lemma \ref{lem:conditions} we obtain that if $\sym{G}$ is maximal, then $n \in \{1,2,3,4,6\}$. We want to reconstruct $\nsym{G} \leq O(\Lambda)$ knowing $\sym{G}$ and $\nsym{G}/\sym{G}\cong \mu_n$. In this situation one speaks of an extension of groups. We are interested not only in the group structure, but also in its action on the K3 lattice. This motivates the next definition.
\begin{definition}
Let $L$ be a lattice, $G \leq O(L)$ a group of isometries and $N \leq G$ a normal subgroup with cyclic quotient $G/N = \langle gN \rangle$. We say that $G$ is an \emph{extension} of $N$ by $g|_{L^N}$ where $L^N$ denotes the invariant lattice of $N\leq O(L)$.
\end{definition}
\begin{rem}
 In our setting $L$ is unimodular and $N$ coincides with the kernel of the natural map $O(L, L_N) \rightarrow O(L_N^\vee/L_N$).
 In this case, $G\leq O(L)$ is uniquely determined by $N$ and $g|_{L^N}$.
\end{rem}

Before extending the group, we first have to extend single elements. We are in the luxurious position that every element extends:
\begin{lem}[{\cite[Thm 5.1]{hashimoto:symplectic}\label{lem:surj}}]
 Let $\Lambda_\sym{G}$ be the coinvariant lattice for one of the $11$ maximal finite groups. Then the natural map
 \[ \psi \colon O(\Lambda_\sym{G}) \rightarrow O(\disc{\Lambda_\sym{G}})\]
 is surjective.
 In particular any isometry of $O(\Lambda^{\sym{G}})$ can be extended to an element in $O(\Lambda)$ normalizing $\sym{G}$.
\end{lem}
\begin{rem}
 One may double check the theorem as follows:
 first compute $O(\Lambda_\sym{G})$ with the Plesken-Souvignier backtracking algorithm.
 Then check by a direct computation that the natural map is surjective. For the reader's convenience we list the orders of the groups involved in Table \ref{tbl:symplectic}. Note that by Lemma \ref{lem:symplectic} we have $\# \sym{G} \cdot \#O(q_{\Lambda_{\sym{G}}}) = \#O(\Lambda_{\sym{G}})$, if and only if the natural map $\psi$ is surjective.
\end{rem}
In general extensions of a given group of isometries are not unique, not even up to conjugacy. But we are in a particulary nice situation.
\begin{lem}\label{lem:unique-extension}
Let $\sym{G} \leq O(\Lambda)$ be one of the $11$ maximal symplectic groups.
Let $g \in O(\Lambda^\sym{G})$ be an isometry. Then
\begin{enumerate}
 \item there is a unique extension of $\sym{G}$ by $g$;
 \item if $\tilde{g}\in O(\Lambda^\sym{G})$ is conjugate to $g$, then the corresponding extensions are conjugate in $O(\Lambda)$.
 \end{enumerate}
\end{lem}
\begin{proof}
 Recall that $\sym{G}$ is a subgroup of the orthogonal group of the K3 lattice $\Lambda$. In particular we have a primitive extension $\Lambda^\sym{G} \oplus \Lambda_\sym{G} \subseteq \Lambda$.
 Since the K3 lattice is unimodular, this primitive extension is determined by an anti-isometry
 \[\phi \colon \disc{{\Lambda^\sym{G}}} \longrightarrow \disc{{\Lambda_\sym{G}}}.\]
 The natural map
 $ \psi \colon O(\Lambda_\sym{G})
 \rightarrow O(\disc{{\Lambda_\sym{G}}})$,
 $f \mapsto \bar f$ is surjective (Lemma \ref{lem:surj}). Hence, we find an $h \in O(\Lambda_\sym{G})$ such that $\overline{h} = \phi \circ \overline{g}\circ \phi^{-1}$.
 This means that $\tilde{g}= g \oplus h$ extends to an isometry of $\Lambda$.
 We set $\nsym{G} = \langle \sym{G},\tilde{g} \rangle$.
 Any other choice of $h$ is of the form $h \cdot (\mathop{id}_{\Lambda^\sym{G}} \oplus f)$ with $f \in \ker \psi \cong \sym{G}$ (Lemma \ref{lem:symplectic} 5.). Then $\nsym{G}$ remains unchanged.
 We now turn to the second claim.
 Let $f\in O(\Lambda^\sym{G})$ and let $g^f= f^{-1}g f$ be a conjugate of $g$. Take an extension $\tilde{g}$ of $g$ to an isometry of $\Lambda$. We can extend $f$ to an isometry $\tilde{f}=f \oplus f'$ of $\Lambda$ as well (Lemma \ref{lem:extend}). Since the restriction $\sym{G}|_{\Lambda_\sym{G}}$ is a normal subgroup of $O(\Lambda_\sym{G})$, conjugation by $f$ preserves $\sym{G}$.
 Further the restriction of $\tilde{g}^{\tilde{f}}$
 to ${\Lambda^\sym{G}}$ is equal to $g^f$.
 Hence, by part 1, the extensions $G^{\tilde{f}}$ and $\langle \tilde{g^f},\sym{G} \rangle$ are equal.
\end{proof}
If $(X,G) \cong (X',G')$ are isomorphic pairs consisting of a $\Lambda$-marked K3 surface with a group of automorphisms, then $G$ and $G'$ (viewed in $O(\Lambda)$ via the marking) are conjugate. In our case the pairs do not deform, so there is hope for the converse statement to hold.
\begin{prop}\label{prop:k3fromgroup}
Let $(X,\eta)$ and $(X',\eta')$ be marked K3 surfaces and $\nsym{G}\leq \Aut(X)$, $\nsym{G'}\leq \Aut(X')$ finite subgroups such that $\sym{G}$ and $\sym{G'}$ are isomorphic to one of the $11$ maximal groups.
Suppose that $\eta G \eta^{-1}$ and $\eta' G' \eta'^{-1}$ are conjugate in $O(\Lambda)$, then there is an isomorphism $f\colon X \rightarrow X'$ with $G = f^{-1}G'f$, i.e. the pairs $(X,G)$ and $(X',G')$ are isomorphic.
\end{prop}
\begin{proof}
 Changing the marking $\eta$ conjugates $\eta G \eta^{-1}$ in $O(\Lambda)$. To ease notation, we identify $G,G'$ with their image in $O(\Lambda)$. In order to use the strong Torelli type Theorem, we have to produce an effective Hodge isometry conjugating $G$ and $G'$.

 Let $n$ be the order of $\nsym{G}/\sym{G}$. We choose a primitive $n$th root of unity $\zeta \in \C$. Then $\nsym{G}/\sym{G}$ comes with a distinguished generator $g\sym{G}$ given by $g\left(\eta_\C(\sigma_X)\right)=\zeta \sigma_X$. And likewise $g'\sym{G'}$.
 By assumption $\nsym{G}$ and $\nsym{G'}$ are conjugate via some $f \in O(\Lambda)$.
 If $n=2$, then the generators $g\sym{G}$ and $g'\sym{G'}$ are unique.
 Otherwise $n=3,4,6$, and then $\SO(\Lambda^{\sym{G}})$ is a dihedral group of order $8$ or $12$ (Lemma \ref{lem:orientation-reversal}). In any case there is a unique conjugacy class
 of order $n$. Since we can extend any conjugator of the dihedral group to an element of $O(\Lambda)$ (Lemma \ref{lem:extend}) preserving $\sym{G}$, we may modify the conjugator $f$ in such a way that it conjugates the distinguished generators $g\sym{G}$ and $g'\sym{G'}$ as well.
 So after conjugation, we may assume that $\nsym{G'} = \nsym{G}$ and further that $g'\sym{G'}  = g \sym{G}$.

 Suppose that $n > 2$. Then the periods of $X$ and $X'$ are uniquely determined by the distingued generators as the (1-dimensional!) eigenspaces with eigenvalue $\zeta$ of $g|H(X,\C)^\nsym{G}$, respectively $g'|H(X',\C)^{\nsym{G'}}$. And we are done. (Note that if $\sigma$ is an eigenvector for $\zeta \neq \pm 1$, then $\langle \sigma,\sigma \rangle=\zeta^2 \langle \sigma,\sigma \rangle$ implies $\sigma^2=0$.)
 If $n=2$ then the eigenspace for $-1$ of $g|H(X,\C)^G$ is of dimension $2$. However, the period is of square zero. Thus the period is one of the two isotropic lines in the eigenspace.
 These correspond to the two orientations of the transcendental lattice.
 By Lemma \ref{lem:orientation-reversal} one can find an isometry $f$ of $\Lambda^{G}$ centralizing $g$ and reversing the orientation. This $f$ extends to an isometry of $\Lambda$ preserving $G$. Thus we have obtained a Hodge isometry conjugating $G$ and $G'$.

 Note that $\H^2(X,\Z)^G$ is spanned by an ample class $l$ and likewise for $G'$. Since our Hodge isometry conjugates $G$ and $G'$ it maps $l$ to $l'$ or $-l'$.
 In the second case our Hodge isometry is not effective. However, we may then replace it by its negative.
\end{proof}

\begin{prop}\label{prop:class}
  Let $\sym{G} \leq O(\Lambda)$ be a maximal symplectic group. There is a one to one correspondence between conjugacy classes of non trivial cyclic subgroups of $\SO(\syminv)$ and isomorphism classes of pairs $(X,\nsym{G}')$ consisting of a K3 surface $X$ and $G' \leq \Aut(X)$ a finite subgroup with $\sym{G}\cong \sym{G'}<G'$.
\end{prop}
\begin{proof}
It remains to show that each cyclic subgroup is actually coming from a K3 surface. To see this, choose a suitable eigenvector of $\nsym{G}/\sym{G}|_{\syminv}$ as period, a generator of $\Lambda^\nsym{G}$ as K\"ahler class which is in fact ample since it is integral. Then use the global Torelli type theorem and surjectivity of the period map.
\end{proof}

\begin{rem}
  Fix a maximal symplectic group $\sym{G} \leq O(\Lambda)$.
  The non-symplectic extensions $\sym{G}<G$ which we are classifying lie in the extension
\begin{equation}
SO(\Lambda^{\sym{G}}).\sym{G} = O(\Lambda)\cap \left(SO(\Lambda^\sym{G}) \times O(\Lambda_\sym{G})\right).
\end{equation}
They are the subgroups $\sym{G} < G\leq SO(\Lambda^{\sym{G}}).\sym{G}$ with $G/\sym{G}$ cyclic.
\end{rem}

\begin{proof}[Proof of Theorem \ref{thm:classification}]
We use the correspondence set up in Proposition \ref{prop:class}.
By Lemma \ref{lem:orientation-reversal}, $\SO(\syminv)$ is isomorphic to a dihedral group $D_n$ of order $2n$ for $n \in \{2,4,6\}$.
Its maximal cyclic subgroups up to conjugacy are two groups of order $2$ generated by reflections and one group of order $n$ generated by a rotation. Thus for each of the $14$ actions there are $3$ maximal extensions leading to $42= 3\cdot 14$ cases.

It remains to derive the additional data which we provide with the tables.
Let $\mu_n \cong \langle g \rangle \leq \SO(\syminv)$ be a cyclic subgroup and $G$ the corresponding extension. The invariant polarization $\Lambda^G$
is computed as the kernel of $1-g \in \End_\Z(\Lambda^G)$ and the transcendental lattice as the orthogonal complement of $\Lambda^G$ in $\Lambda^{\sym{G}}$.

To obtain the group structure of $G$, we construct $\Lambda$ as a primitive extension of $\Lambda^{\sym{G}} \oplus \Lambda_{\sym{G}}$
by calculating the corresponding glue map.
Then we extend $g$ to an isometry $g\oplus h$ of $\Lambda$
as in Lemmas \ref{lem:extend} and \ref{lem:unique-extension}.
Here the coinvariant lattices $\Lambda_{\sym{G}}$ are obtained as sublattices of the Leech lattice as tabulated in \cite{hoehn:leech}
and the invariant lattices $\Lambda^{\sym{G}}$ are tabulated in \cite{hashimoto:symplectic}.
The gluings and extensions are carried out using the code developed by the first author for sageMath \cite{sage}.
We showcase the computation for \textbf{Nos.\ 70a, 70b, 70c} with a notebook in the ancillary files.
\end{proof}

\section{The classification}\label{sec:classification}
Using Proposition \ref{prop:class}, we are ready to state the details of the classification. The tables were produced using SageMath \cite{sage} and GAP \cite{GAP4}.
We denote by $\Z l=\Lambda^G$
the (primitive) invariant polarization of $\nsym{G}$. Then
$\nsym{G}:=\operatorname{Aut}(X,l)$ is the full projective automorphism group. The lattice $\syminv$ is the fixed lattice $H^2(X,\Z)^\sym{G}$ of a maximal symplectic action.
The entry ``glue'' denotes the index $[\syminv:T_X \oplus \Z l]$.
The GAP Id \cite{eick:small-groups} identifies a group up to isomorphism. If equations for the pair $(X,G)$ are known, then we write its identifier in bold.
We set $\zeta_n=\exp(2 \pi \sqrt{-1}/n)$, $\omega=\zeta_3$ and $i=\zeta_4$.

\begin{rem}
 We note that in all cases, except \textbf{62b} with $\Z l \cong (12)$, the exact sequence \[1 \rightarrow \sym{G} \rightarrow \nsym{G} \rightarrow \mu_n \rightarrow 1\]
 splits.
Namely, $G$ is a semidirect product of $\sym{G}$ and $\mu_n$.
\end{rem}

\subsection{No.\ 54} \label{subsect:T48}
$\sym{G}=T_{48}$.
We have $\SO(\syminv) \cong D_{6}$.
\begin{equation*}
\begin{array}{c|c|c|c|c|c|c}
  \hline \Tstrut
  \syminv & n & T_X & \Z l & \text{glue} & \text{GAP Id} & \text{case} \\ \hline

  & 6 & \begin{pmatrix} 16 &  8 \\  8 & 16 \end{pmatrix} & (2) & 1 & [288,900] 
& \textbf{54a} \\
  \smash{\begin{pmatrix}
            2 & 0 & 0 \\
            0 & 16 & 8 \\
            0 & 8 & 16
  \end{pmatrix}}
   & 2 & \begin{pmatrix}  2 &  0 \\  0 & 48 \end{pmatrix} & (16) & 2 &[96, 193] & \text{54b} \\
  & 2 & \begin{pmatrix}  2 &  0 \\  0 &  16 \end{pmatrix} & (48) & 2 & [96,193] & \text{54c} \\ \hline
\end{array}
\end{equation*}

A projective model of \textbf{54a} is given in \cite{mukai:symplectic}.
It is the double cover $X$ of $\P^2$ branched over the curve defined by
\begin{equation}
 xy(x^4+y^4)+z^6=0. %\text{~in~} \P^2
\end{equation}
We have $\nsym{G}=G_s \times \mu_6$, where
 $\mu_6$ is generated by a lift to $X$ of
% colon
% $(x,y,z)\mapsto (x,y,\zeta_6 z)$.
% comma
 $(x : y : z)\mapsto (x : y : \zeta_6 z)$.

\subsection{No.\ 62}
$\sym{G}=N_{72}$.
We have $\SO(\syminv) \cong D_4$.
\begin{equation*}
\begin{array}{c|c|c|c|c|c|c}
  \hline \Tstrut
  \syminv & n & T_X & \Z l & \text{glue} & \text{GAP Id} & \text{case} \\ \hline

  & 2 & \begin{pmatrix} 6 &  0 \\  0 & 36 \end{pmatrix} & (6) & 2 & [144,186] & \textbf{62a} \\
  \smash{\begin{pmatrix}
            6 & 0 & 3 \\
            0 & 6 & 3 \\
            3 & 3 & 12
  \end{pmatrix}}
   & 2 & \begin{pmatrix}  12 &  6 \\  6 & 12 \end{pmatrix} & (12) & 2 & [144, 182] & \text{62b} \\
  & 4 & \begin{pmatrix}  6 &  0 \\  0 &  6 \end{pmatrix} & (36) & 2 & [288, 841] & \textbf{62c} \\ \hline
\end{array}
\end{equation*}

A projective model of \textbf{62a} is given in \cite{mukai:symplectic}:
\begin{equation}
 x_1^3+x_2^3+x_3^3+x_4^3=x_1 x_2+x_3 x_4+x_5^2=0
 \quad\text{in}\quad \P^4.
\end{equation}
We have $\nsym{G}=\sym{G} \times \mu_2$, where
 $\mu_2$ is generated by
% comma
% $(x_1,\ldots,x_4,x_5) \mapsto (x_1,\ldots,x_4,-x_5)$.
% colon
 $(x_1 : \cdots : x_4 : x_5) \mapsto (x_1 : \cdots : x_4 : -x_5)$.

 A projective model of \textbf{62c} in $\mathbb{P}^5$ together with a non-linear action of
 $N_{72} \times \mu_2$ and an invariant polarization of degree $36$ is given in \cite{mukai-ohashi:m12}:
 \[X=X_{\lambda,\mu}=
 \begin{cases}
 x_0^2 - \lambda x_1x_2 = y_0^2 - \mu y_1y_2 \\
 x_1^2 - \lambda x_0x_2 = y_1^2 - \mu y_0y_2 \\
 x_2^2 - \lambda x_0x_1 = y_2^2 - \mu y_0y_1 \\
 \end{cases} (\lambda = 1 + \sqrt{3}, \mu = 1-\sqrt{3})\]
 It remains to exhibit an element acting by a primitive $4$th root of unity on the $2$-form \cite{ohashi:private}:
 the variety is constructed using the partial derivatives of the the Hesse pencil
 \[z_0^3 + z_1^3 + z_2^3 - 3\kappa z_0 z_1 z_2.\]
 A linear map $g \in \PGL_3(\mathbb{C})$ preserving this pencil acts on the base as a M\"obius transformation, which is denoted by the same letter $g$.
 By \cite[Lemma 2.1]{mukai-ohashi:m12} it induces a morphism $\tilde{g} \colon X_{\lambda,\mu} \rightarrow X_{g(\lambda),g(\mu)}$. Take
\begin{equation*}
g=\left(\begin{matrix}
   1 & 1 & \omega^2 \\
   1 & \omega & \omega \\
   \omega  & 1& \omega
  \end{matrix}\right).\quad \mbox{Then }  \; g(\kappa)=\frac{\kappa+2\omega^2}{\omega \kappa-1}
\end{equation*}
and $g$ exchanges $\lambda$ and $\mu$.
We compose $\tilde{g}\colon X_{\lambda,\mu} \rightarrow X_{\mu,\lambda}$ with the map $\iota \colon X_{\mu,\lambda} \rightarrow X_{\lambda, \mu}$ exchanging the $x_i$ and $y_i$ to obtain the sought for automorphism of $X$,
which generates $\mu_4$ with $G=N_{72} \rtimes \mu_4$.
Note that $\mu_2$ above, which is generated by $y_i \mapsto -y_i$,
is not included in $\mu_4$.
Let $h$ be the class of hyperplane section and let $f_\infty$ be the class represented by the smooth elliptic curve of degree $6$ defined by
\begin{equation}
 \rank A \leq 1, \quad\text{where}\quad A:=
 \begin{pmatrix}
  \mu x_0 & x_2 + c y_2 & x_1 - c y_1 \\
  x_2 - c y_2 & \mu x_1 & x_0 + c y_0 \\
  x_1 + c y_1 & x_0 - c y_0 & \mu x_2 
 \end{pmatrix}, ~ c^2=1-\mu^2.
\end{equation}
Then we have $h,f_\infty \in \NS_X^H$, where $H \cong (C_3^2 \rtimes C_4) \rtimes \mu_4$ is the subgroup of $G$ consisting of all linear transformations.
Moreover, $3h-f_\infty$ is a $G$-invariant polarization of degree $36$ (see \cite{mukai-ohashi:m12} for details).

% 62c should be N_72 example of
% indeed the action is semi symplectic and has a unique
% invariant polarization of degree 36 (on the enriques that is 18)
% but there is an order 4 non-symplectic automorphism missing
% and that is 79f as well, the A_6 example!
% which is the one given by oguiso zhang and it is in the appendix of mukai-ohashi
\subsection{No.\ 63}
$\sym{G}=M_9$.
We have $\SO(\syminv) \cong D_{6}$.
\begin{equation*}
\begin{array}{c|c|c|c|c|c|c}
  \hline \Tstrut
  \syminv & n & T_X & \Z l & \text{glue} & \text{GAP Id} & \text{case} \\ \hline

  & 6 & \begin{pmatrix} 12 &  6 \\  6 & 12 \end{pmatrix} & (2) & 1 & [432, 735] & \textbf{63a} \\
  \smash{\begin{pmatrix}
            2 & 0 & 0 \\
            0 & 12 & 6 \\
            0 & 6 & 12
  \end{pmatrix}}
   & 2 & \begin{pmatrix}  2 &  0 \\  0 & 36 \end{pmatrix} & (12) & 2 & [144,182] & \text{63b} \\
  & 2 & \begin{pmatrix}  2 &  0 \\  0 &  12 \end{pmatrix} & (36) & 2 & [144,182] & \text{63c} \\ \hline
\end{array}
\end{equation*}

A projective model of \textbf{63a} is given in \cite{mukai:symplectic}.
It is the double cover $X$ of $\P^2$ branched over the curve defined by
\begin{equation}
 x^6+y^6+z^6-10(x^3 y^3+y^3 z^3+z^3 x^3)=0.
\end{equation}
We have $\nsym{G}=\sym{G} \rtimes \mu_6$, where
 $\mu_6$ is generated by the covering transformation and
 a lift to $X$ of
% comma
% $(x,y,z) \mapsto (\zeta_3 x,y,z)$.
% colon
 $(x : y : z) \mapsto (\zeta_3 x : y : z)$.

\subsection{No.\ 70}
$\sym{G}=\MF{S}_5$.
We have $\SO(\syminv) \cong D_2$ in both cases.
Since the center of $\sym{G}$ is trivial and there is no nontrivial outer-automorphism of $\sym{G}$, we have $\nsym{G}=\sym{G} \times \mu_2$ in each case.
\begin{equation*}
\begin{array}{c|c|c|c|c|c|c}
  \hline \Tstrut
  \syminv & n & T_X & \Z l & \text{glue} & \text{GAP Id} & \text{case} \\ \hline
  & 2 & \begin{pmatrix} 10 &  0 \\  0 & 20 \end{pmatrix} & (6) & 2 & [240,189] & \textbf{70a} \\
  \smash{\begin{pmatrix}
            4 & 1 & 0 \\
            1 & 4 & 0 \\
            0 & 0 & 20
  \end{pmatrix}}
   & 2 & \begin{pmatrix}  6 &  0 \\  0 & 20 \end{pmatrix} & (10) & 2 & [240,189] & \textbf{70b} \\
  & 2 & \begin{pmatrix}  4 &  1 \\  1 &  4 \end{pmatrix} & (20) & 1 & [240,189] & \textbf{70c} \\ \hline
  & 2 & \begin{pmatrix} 20 &  10 \\  10 & 20 \end{pmatrix} & (4) & 2 & [240,189] & \textbf{70d} \\
  \smash{\begin{pmatrix}
            4 & 2 & 2 \\
            2 & 6 & 1 \\
            2 & 1 & 16
  \end{pmatrix}}
   & 2 & \begin{pmatrix}  4 &  2 \\  2 & 16 \end{pmatrix} & (20) & 2 & [240,189] & \text{70e} \\
  & 2 & \begin{pmatrix}  4 &  2 \\  2 &  6 \end{pmatrix} & (60) & 2 & [240,189] & \textbf{70f} \\ \hline
\end{array}
\end{equation*}

A projective model of \textbf{70a} is given in \cite{mukai:symplectic}:
\begin{equation}
 \sum_{i=1}^5 x_i=\sum_{i=1}^6 x_i^2=\sum_{i=1}^5 x_i^3=0
 \quad\text{in}\quad \P^5,
\end{equation}
 where $G=\mathfrak{S}_5 \times \mu_2$ is generated by the permutations of $x_1,\ldots,x_5$ and
% comma
% $(x_1,\ldots,x_5,x_6) \mapsto (x_1,\ldots,x_5,-x_6)$.
% colon
 $(x_1 : \cdots : x_5 : x_6) \mapsto (x_1 : \cdots : x_5 : -x_6)$.

A projective model of \textbf{70b} is given in \cite[Proposition 3.16]{frantzen:automorphisms}.
%For equations of see Section \ref{sect:70b}.
Let $X$ be the minimal resolution of the double cover of $\mathbb{P}^2$ branched over the curve $C$ defined by $f=0$, where
\begin{align*}
f &=
 2 ( x_1^4 x_2 x_3 + x_1 x_2^4 x_3 + x_1 x_2 x_3^4 ) 
 -2 ( x_1^4 x_2^2 + x_1^4 x_3^2 + x_1^2 x_2^4 + x_1^2 x_3^4 + x_2^4 x_3^2
  + x_2^2 x_3^4 ) \\
 & +2 ( x_1^3 x_2^3 + x_1^3 x_3^3 + x_2^3 x_3^3 )
  + ( x_1^3 x_2^2 x_3 + x_1^3 x_2 x_3^2 + x_1^2 x_2^3 x_3
  + x_1^2 x_2 x_3^3 + x_1 x_2^3 x_3^2 + x_1 x_2^2 x_3^3 ) \\
 & -6 x_1^2 x_2^2 x_3^2.
\end{align*}
Let $H \cong \mathfrak{S}_4$ be the subgroup of $\PGL_3(\mathbb{C})$ permuting the following four points:
\begin{equation}
 (1:0:0),~(0:1:0),~(0:0:1),~(1:1:1).
\end{equation}
Then the curve $C$ is preserved by $H$ and the following Cremona transformation (of order $5$):
\begin{equation}
% comma
% g \colon (x_1,x_2,x_3) \mapsto ( x_1 (x_3-x_2),x_3 (x_1-x_2),x_1 x_3 ).
% colon
 g \colon (x_1 : x_2 : x_3) \mapsto ( x_1 (x_3-x_2) : x_3 (x_1-x_2) : x_1 x_3 ).
\end{equation}
The actions of $H$ and $g$ lift to those on $X$.
The group $G \cong \mathfrak{S}_5 \times \mu_2$ is generated by
 $H$, $g$ and the covering transformation.
Note that $X/\mu_2$ is isomorphic to the Del Pezzo surface of degree $5$ (see \cite{frantzen:automorphisms} for details).
From this fact, we know the following:
consider the line defined by $x_1=0$ in $\mathbb{P}^2$ and let $R_1$ be the strict transform of its pull-back in $X$.
Then the orbit of $R_1$ (in $\NS_X$) under the action of $G$ consists of ten curves $R_i$ ($1 \leq i \leq 10$), whose dual configuration is the Petersen graph with $R_i.R_i=-2$ and $R_i.R_j \in \{ 0,2 \}$ for $i \neq j$ (see also \cite{frantzen:automorphisms}).
Hence we have
\begin{equation}
 R:=\sum_{i=1}^{10} R_i, \quad R.R=10 \cdot (-2) + 2 \cdot 15 \cdot 2=40.
\end{equation}
By construction $(X,G)$ is among the surfaces in the table. Further $R$ is invariant
under $G$. Thus $R$ must lie in $\Z l$. This is only possible if
% $R$ is divisible by $2$ as an element in $\NS_X$ and
$R = 2 l$ and $l^2=10$. Thus $(X,G)$ is a projective model of \textbf{70b}.

Projective models of \textbf{70c} are given in \cite[Type VI]{kondo:enriques} and
\cite{peters:melas-codes}.
We review the one in \cite{peters:melas-codes}.
Let $Y$ be the surface defined by the following equations:
\begin{equation}
 \sum_{i=1}^5 x_i=
%\sum_{i=1}^5 \prod_{j \in \{ 1,\ldots,5 \} \setminus \{ i \} } x_j =0
 \sum_{i=1}^5 \frac{1}{x_i}=0
 \quad\text{in}\quad \P^4
\end{equation}
and let $X$ be the minimal resolution of $Y$, which is a K3 surface.
It follows from \cite[Thm 4.15]{hashimoto:family} that $X$ has the same transcendental lattice as \textbf{70c}.
The symmetric group $\MF{S}_5$ acts on $X$ by permutation of $x_i$.
Moreover, the involution
 $\epsilon \colon (x_i) \mapsto (1/x_i)$ acts on $X$.
%
%The surface $Y$ has $10$ singular points of type $A_1$ at e.g.\ $[1,-1,0,0,0]$,
% and we have the corresponding exceptional divisors $D_k$ ($1 \leq k \leq 10$) on $X$.
%By $\iota$, these divisors map to lines $C_{ij}$ defined by $x_i=x_j=0$.
%We have
%\begin{gather}
% D:=\sum_{k=1}^{10} D_k, \quad
% C:=\sum_{1 \leq i<j \leq 5} C_{ij}, \quad
% \iota^* D=C, \quad \iota^* C=D, \\
% D^2=C^2=10 \cdot (-2)=-20, \quad C.D=10 \cdot 3=30.
%\end{gather}
%Therefore $l:=C+D$ is invariant under the action of $\MF{S}_5 \times \mu_2$ and $l^2=20$.
%Suppose we are in Case 70e. Then there should be some $v\in T_X$
% such that $(v+l)/2 \in \Lambda$.
%However, we have
%\begin{equation}
%D_k.(v+l)/2=D_k.l/2=(D_k.C+D_k.D)/2=(3-2)/2=1/2 \not\in \Z,
%\end{equation}
% which is a contradiction.
%Thus, we must be in Case 70c.
The group $G \cong \MF{S}_5 \times \mu_2$ is generated by $\MF{S}_5$ and $\epsilon$.

A projective model of \textbf{70d} is given in \cite[Thm 4.15]{hashimoto:family}.
Consider $Q := \P^1 \times \P^1$ defined by
 $\sum_{i=1}^{5} x_i=\sum_{i=1}^{5} x_i^2=0$ in $\P^4$.
Let $X$ be the double cover of $Q$ branched over the curve defined by $\sum_{i=1}^{5} x_i^4=0$.
Then $X$ is a K3 surface and $G$ is generated by the permutations of $x_i$ and the covering transformation.
This projective model is a (degenerate) member of the family XIII in \cite{smith}.
The pull-back to $Q$ of the hyperplane class $h$ of $\P^4$ is of bidegree $(1,1)$, which has self intersection $2$. Hence the pull-back to $X$ of $h$ has self intersection $2 \cdot 2=4$ and is invariant under $G$. In particular it must be equal to $l$, confirming that we are indeed in case \textbf{70d}.
Note that the linear system $|l|$ is hyperelliptic.

Projective models of \textbf{70f} are given in \cite[Type VII]{kondo:enriques} and \cite[1.2]{mukai-ohashi:m12}.
We review the one in \cite{mukai-ohashi:m12},
 which is given as the minimal resolution of the singular surface
\[\overline{X}\colon \quad \sum_{1\leq i < j \leq 5} x_ix_j = \sum_{1\leq i < j \leq 5} \frac{1}{x_ix_j}=0\]
in $\P^4$. The symmetric group $\MF{S}_5$ acts (non-symplectically) by permutations of the coordinates and a non-symplectic involution is given by the Enriques involution $\epsilon \colon (x_i) \mapsto (1/x_i)$.
The surface $\overline{X}$ has five $A_1$ singularities in the coordinate points. Denote by $E_i$ the corresponding exceptional divisors on $X$. Let $R_i$ be the strict transform of $\overline{X} \cap \{x_i=0\}$. Then the intersection numbers are given by $R_i.C_j=2(1-\delta_{ij})$.
%Set $W_i=E_i + R_i$.
An invariant polarization of degree $60$ is given by $h=\sum_{i=1}^5 (E_i+R_i)$.

\subsection{No.\ 74}
$\sym{G}=L_2(7)$.
We have $\SO(\syminv) \cong D_2$, $D_4$, respectively, as follows.
\begin{equation*}
\begin{array}{c|c|c|c|c|c|c}
  \hline \Tstrut
  \syminv & n & T_X & \Z l & \text{glue} & \text{GAP Id} & \text{case} \\ \hline
  & 2 & \begin{pmatrix} 14 &  0 \\  0 & 28 \end{pmatrix} & (2) & 2 & [336,209] & \textbf{74a} \\
  \smash{\begin{pmatrix}
            2 & 1 & 0 \\
            1 & 4 & 0 \\
            0 & 0 & 28
  \end{pmatrix}}
   & 2 & \begin{pmatrix}  2 &  0 \\  0 & 28 \end{pmatrix} & (14) & 2 & [336,208] & \text{74b} \\
  & 2 & \begin{pmatrix}  2 &  1 \\  1 &  4 \end{pmatrix} & (28) & 1 & [336,208] & \textbf{74c} \\ \hline
  & 4 & \begin{pmatrix} 14 &  0 \\  0 & 14 \end{pmatrix} & (4) & 2 & [672,1046] & \textbf{74d} \\
  \smash{\begin{pmatrix}
            4 & 2 & 2 \\
            2 & 8 & 1 \\
            2 & 1 & 8
  \end{pmatrix}}
   & 2 & \begin{pmatrix}  4 &  0 \\  0 & 14 \end{pmatrix} & (14) & 2 & [336,208] & \text{74e} \\
  & 2 & \begin{pmatrix}  4 &  2 \\  2 &  8 \end{pmatrix} & (28) & 2 &[336,208] & \text{74f} \\ \hline
\end{array}
\end{equation*}

A projective model of \textbf{74a} is given in \cite{oguiso:168-k3}.
It is the double cover of $\mathbb{P}^2$ branched over the Hessian of the Klein curve defined by
\begin{equation}
 x^5 z + y^5 x + z^5 y - 5 x^2 y^2 z^2=0.% \quad\text{in}\quad \P^2.
\end{equation}
We have $\nsym{G}=L_2(7) \times \mu_2$, where $\mu_2$ is generated by the covering transformation.

A projective model of \textbf{74c} is given in \cite{ujikawa:discr7} and \cite[Appendix]{artebani_sarti_taki:non-symplectic}.
It is the universal elliptic curve over $X_1(7)$.
We have $\nsym{G}=PGL_2(\mathbb{F}_7)=L_2(7) \rtimes \mu_2$.

A projective model of \textbf{74d} is given in \cite{mukai:symplectic}.
Using the Klein curve with $L_2(7)$, it is defined by
\begin{equation}
 x^3 y+y^3 z+z^3 x+w^4=0 \quad\text{in}\quad \P^3.
\end{equation}
We have $\nsym{G}=L_2(7) \times \mu_4$, where $\mu_4$ is generated by
% comma
% $(x,y,z,w) \mapsto (x,y,z,i w)$.
% colon
 $(x : y : z : w) \mapsto (x : y : z : i w)$.

\subsection{No.\ 76}
$\sym{G}=H_{192}$.
We have $\SO(\syminv) \cong D_2$.
\begin{equation*}
\begin{array}{c|c|c|c|c|c|c}
  \hline \Tstrut
  \syminv & n & T_X & \Z l & \text{glue} & \text{GAP Id} & \text{case} \\ \hline

  & 2 & \begin{pmatrix} 8 &  0 \\  0 & 12 \end{pmatrix} & (4) & 1 &[384,17948] & \textbf{76a} \\
  \smash{\begin{pmatrix}
            4 & 0 & 0 \\
            0 & 8 & 0 \\
            0 & 0 & 12
  \end{pmatrix}}
   & 2 & \begin{pmatrix}  4 &  0 \\  0 & 12 \end{pmatrix} & (8) & 1 &[384,17948] & \textbf{76b} \\
  & 2 & \begin{pmatrix}  4 &  0 \\  0 &  8 \end{pmatrix} & (12) & 1 &[384,17948] & \textbf{76c} \\ \hline
\end{array}
\end{equation*}

A projective model of \textbf{76a} is given as follows
(this is a (degenerate) member of the family XI in \cite{smith}):
consider $Q \cong \P^1 \times \P^1$ defined by
$\sum_{i=1}^{4} x_i^2=0$ in $\P^3$.
Let $X$ be the double cover of $Q$ branched over the curve defined by $\sum_{i=1}^{4} x_i^4=0$.
Then $X$ is a K3 surface and $G=H_{192} \times \mu_2$ is generated by the permutations of $x_i$, the sign changes of $x_i$, and the covering transformation.
Similarly to \textbf{70d}, we have $l^2=4$.
Namely, since the pull-back to $X$ of the hyperplane class of $\P^3$ has self intersection $4$ and is invariant under $G$,
the pair $(X,G)$ is case \textbf{76a}.

A projective model of \textbf{76b} is given in \cite{mukai:symplectic}:
\begin{equation}
 x_1^2+x_3^2+x_5^2=x_2^2+x_4^2+x_6^2, \quad
 x_1^2+x_4^2=x_2^2+x_5^2=x_3^2+x_6^2 \quad\text{in}\quad \P^5.
\end{equation}
We have $\nsym{G}=H_{192} \times \mu_2$, where $\mu_2$ is generated by
% comma
% $(x_1,\ldots,x_6) \mapsto (-x_1,x_2,-x_3,x_4,-x_5,x_6)$.
% colon
 $(x_1 : \cdots : x_6) \mapsto (-x_1 : x_2 : -x_3 : x_4 : -x_5 : x_6)$.

A projective model of \textbf{76c} is given as a smooth surface $X$ of tri-degree $(2,2,2)$ in $\P^1 \times \P^1 \times \P^1$ \cite[Example 6]{mukai-ohashi:m12}:
\begin{equation}
 v^2w^2 + u^2w^2+u^2v^2 + 1 + i \left(u^2+v^2+w^2 +u^2v^2w^2\right)=0
\end{equation}
where $(u,v,w)$ are affine coordinates on $(\P^1)^3$.
The surface $X$ admits the following linear actions: the permutations of $(u,v,w)$, $(u,v,w) \mapsto (\pm u, \pm v, \pm w)$, and $(u,v,w) \mapsto (iu,iv,i/w)$.
Those linear actions generate $G=H_{192} \times \mu_2$. Here $\mu_2$ is generated by $(u,v,w) \mapsto (-u,-v,-w)$. The invariant polarization of degree $12$ is given by $\mathcal{O}_X(1,1,1)$.

\subsection{No.\ 77}
$\sym{G}=T_{192}$.
We have $\SO(\syminv) \cong D_{6}$.
\begin{equation*}
\begin{array}{c|c|c|c|c|c|c}
  \hline \Tstrut
  \syminv & n & T_X & \Z l & \text{glue} & \text{GAP Id} & \text{case} \\ \hline

  & 6 & \begin{pmatrix} 8 &  4 \\  4 & 8 \end{pmatrix} & (4) & 1 &[1152,157515] & \textbf{77a} \\
  \smash{\begin{pmatrix}
            4 & 0 & 0 \\
            0 & 8 & 4 \\
            0 & 4 & 8
  \end{pmatrix}}
   & 2 & \begin{pmatrix}  4 &  0 \\  0 & 24 \end{pmatrix} & (8) & 2 & [384,5602] & \text{77b} \\
  & 2 & \begin{pmatrix}  4 &  0 \\  0 &  8 \end{pmatrix} & (24) & 2 & [384,5608] & \text{77c} \\ \hline
\end{array}
\end{equation*}

A projective model of \textbf{77a} is given in \cite{mukai:symplectic}:
\begin{equation}
 x^4+y^4+z^4+w^4-2\sqrt{-3}(x^2 y^2+z^2 w^2)=0 \quad\text{in}\quad \P^3.
\end{equation}
We have $\nsym{G}=(T_{24}*T_{24})\rtimes \langle \tau,\sigma \rangle=T_{192} \rtimes \mu_6$, where
 $*$ denotes central product,
 $T_{24}$ the binary tetrahedral group,
 $\tau$ the involution interchanging two copies of $T_{24}$ and $\sigma$ switches the sign of $x$.

\subsection{No.\ 78}
$\sym{G}=\MF{A}_{4,4}$.
We have $\SO(\syminv) \cong D_4$.
\begin{equation*}
\begin{array}{c|c|c|c|c|c|c}
  \hline \Tstrut
  \syminv & n & T_X & \Z l & \text{glue} & \text{GAP Id} & \text{case} \\ \hline

  & 4 & \begin{pmatrix} 12 &  0 \\  0 & 12 \end{pmatrix} & (8) & 2 & [1152, 157850] & \textbf{78a} \\
  \smash{\begin{pmatrix}
            8 & 4 & 4 \\
            4 & 8 & 2 \\
            4 & 2 & 8
  \end{pmatrix}}
   & 2 & \begin{pmatrix}  8 &  0 \\  0 & 12 \end{pmatrix} & (12) & 2 &[576, 8654] & \text{78b} \\
  & 2 & \begin{pmatrix}  8 &  4 \\  4 &  8 \end{pmatrix} & (24) & 2 &[576, 8653] & \text{78c} \\ \hline
\end{array}
\end{equation*}
A projective model of \textbf{78a} is given in \cite{mukai:symplectic}:
\begin{equation}
 \begin{pmatrix}
  1 & 1 & 1 \\
  1 & \omega & \omega^2 \\
  1 & \omega^2 & \omega
 \end{pmatrix}
 \begin{pmatrix} x^2 \\ y^2 \\ z^2 \end{pmatrix}
 =
 \sqrt{3}
 \begin{pmatrix} u^2 \\ v^2 \\ w^2 \end{pmatrix}
 \quad\text{in}\quad \P^5.
\end{equation}
We have $\nsym{G}=\MF{A}_{4,4} \rtimes \mu_4$, where $\mu_4$ is generated by
% comma
% $(x,y,z,u,v,w)\mapsto(u,v,w,x,z,y)$.
% colon 
$(x : y : z : u : v : w)\mapsto(u : v : w : x : z : y)$.

\subsection{No.\ 79}
$\sym{G}=\MF{A}_6$.
We have $\SO(\syminv) \cong D_2$, $D_4$, respectively, as follows.

\begin{equation*}
\begin{array}{c|c|c|c|c|c|c}
  \hline \Tstrut
  \syminv & n & T_X & \Z l & \text{glue} & \text{GAP Id} & \text{case} \\ \hline
  & 2 & \begin{pmatrix} 12 &  0 \\  0 & 30 \end{pmatrix} & (2) & 2 &[720, 766] & \textbf{79a} \\
  \smash{\begin{pmatrix}
            2 & 1 & 0 \\
            1 & 8 & 0 \\
            0 & 0 & 12
  \end{pmatrix}}
   & 2 & \begin{pmatrix}  2 &  1 \\  1 & 8 \end{pmatrix} & (12) & 1& [720, 764] & \text{79b} \\
  & 2 & \begin{pmatrix}  2 &  0 \\  0 &  12 \end{pmatrix} & (30) & 2 &[720, 764] & \text{79c} \\ \hline
  & 2 & \begin{pmatrix} 6 &  0 \\  0 & 20 \end{pmatrix} & (6) & 2 & [ 720, 763] & \textbf{79d} \\
  \smash{\begin{pmatrix}
            6 & 0 & 3 \\
            0 & 6 & 3 \\
            3 & 3 & 8
  \end{pmatrix}}
   & 2 & \begin{pmatrix}  8 &  2 \\  2 & 8 \end{pmatrix} & (12) & 2 & [ 720, 764 ] & \text{79e} \\
  & 4 & \begin{pmatrix}  6 &  0 \\  0 &  6 \end{pmatrix} & (20) & 2 & [ 1440, 4595 ] & \textbf{79f} \\ \hline
\end{array}
\end{equation*}

A projective model $X$ of \textbf{79a} is given as follows (see also \cite[p.14 (L), p.18]{smith}).
Consider the invariant curve of degree $6$ by the Valentiner group in
 $\GL_3(\C)$, which is defined in $\P^2$ by the following equation
 \cite{elkies}:
 \[10 x^3 y^3 + 9 (x^5+y^5)z - 45 x^2 y^2 z^2 - 135 xyz^4 +27 z^6=0.\]
The K3 surface $X$ is defined as the double cover branched over this curve.
We have $\nsym{G}=\MF{A}_6 \times \mu_2$, where $\mu_2$ is generated by the covering transformation.

A projective model $X$ of \textbf{79d} is given in \cite{mukai:symplectic}:
\begin{equation}
 \sum_{i=1}^6 x_i=\sum_{i=1}^6 x_i^2=\sum_{i=1}^6 x_i^3=0
 \quad\text{in}\quad \P^5.
\end{equation}
The symmetric group $\MF{S}_6$ of degree $6$ acts on $X$.
The hyperplane section of $X$ has self intersection $6$ and is invariant under $\MF{S}_6$.
Hence $l^2=6$ and $\nsym{G}=\MF{S}_6=\MF{A}_6 \rtimes \mu_2$.

In \textbf{79f}, $\nsym{G}/\mu_2=\sym{G} \rtimes (\mu_4/\mu_2)$ is isomorphic to $M_{10}$ \cite{oguiso_keum_zhang:alternating}.
A projective model, which is same as \textbf{62c}, together with a (non-linear) action by $\mathfrak{A}_6 \times \mu_2$ and an invariant polarization of degree $20$ is given in \cite{mukai-ohashi:m12}. The missing automorphism acting by a primitive $4$th root of unity is the one given in \textbf{62c}.
The invariant polarization is given by $h+f_\infty$ using the same notation as in \textbf{62c}.

The full groups of automorphisms for \textbf{79b}, \textbf{79c} and \textbf{79f} are calculated in \cite{shimada:alternating}.

\subsection{No. 80}
$\sym{G}=F_{384}$.
We have $\SO(\syminv) \cong D_4$.
\begin{equation*}
\begin{array}{c|c|c|c|c|c|c}
  \hline \Tstrut
  \syminv & n & T_X & \Z l & \text{glue} & \text{GAP Id} & \text{case} \\ \hline

  & 4 & \begin{pmatrix} 8 &  0 \\  0 & 8 \end{pmatrix} & (4) & 1 & [ 1536 , 408544807 ] & \textbf{80a} \\
  \smash{\begin{pmatrix}
            4 & 0 & 0 \\
            0 & 8 & 0 \\
            0 & 0 & 8
  \end{pmatrix}}
   & 2 & \begin{pmatrix}  4 &  0 \\  0 & 8 \end{pmatrix} & (8) & 1 &[ 768, 1090134 ] & \textbf{80b} \\
  & 2 & \begin{pmatrix}  4 &  0 \\  0 &  16 \end{pmatrix} & (16) & 2 & [ 768, 1086051 ] & \text{80c} \\ \hline
\end{array}
\end{equation*}

A projective model of \textbf{80a} is given in \cite{mukai:symplectic}:
\begin{equation}
 x^4+y^4+z^4+w^4=0 \quad\text{in}\quad \P^3.
\end{equation}
We have $\nsym{G}=F_{384} \rtimes \mu_4$, where
$\mu_4$ is generated by
% comma
% $(x,y,z,w) \mapsto (i x,y,z,w)$.
% colon
 $(x : y : z : w) \mapsto (i x : y : z : w)$.

A projective model of \textbf{80b} \cite{bonnaf-sarti:personal} is given by
\begin{align}
   q_1&= \phantom{+}2 x_2^2 - x_3^2 - i x_4^2 + i x_5^2 - x_6^2=0\\
   q_2&= -x_1^2 - i x_2^2 - x_3^2 + i x_4^2 + 2 x_5^2 =0\\
   q_3&= -x_1^2 + i x_2^2         + 2 x_4^2 - i x_5^2 - x_6^2=0
\end{align}
with linear action generated by
\begin{equation}
 \left(\begin{smallmatrix}
-1 & 0 & 0 & 0 & 0 & 0 \\
0 & -j^{2} & 0 & 0 & 0 & 0 \\
0 & 0 & 0 & 0 & j & 0 \\
0 & 0 & 0 & 0 & 0 & -j \\
0 & 0 & 0 & 1 & 0 & 0 \\
0 & 0 & j^{2} & 0 & 0 & 0
\end{smallmatrix}\right) \mbox{ and }
\left(\begin{smallmatrix}
0 & 0 & 0 & 0 & 1 & 0 \\
0 & 0 & 0 & 0 & 0 & 1 \\
0 & 0 & 0 & 1 & 0 & 0 \\
0 & 0 & 1 & 0 & 0 & 0 \\
-1 & 0 & 0 & 0 & 0 & 0 \\
0 & 1 & 0 & 0 & 0 & 0
\end{smallmatrix}\right)
\end{equation}
where $j^2 = i$.

\subsection{No.\ 81}
$\sym{G}=M_{20}$.
We have $\SO(\syminv) \cong D_4$.
\begin{equation*}
\begin{array}{c|c|c|c|c|c|c}
  \hline \Tstrut
  \syminv & n & T_X & \Z l & \text{glue} & \text{GAP Id} & \text{case} \\ \hline

  & 2 & \begin{pmatrix} 4 &  0 \\  0 & 40 \end{pmatrix} & (4) & 2  &[ 1920, 240993 ] & \textbf{81a} \\
  \smash{\begin{pmatrix}
            4 & 0 & 2 \\
            0 & 4 & 2 \\
            2 & 2 & 12
  \end{pmatrix}}
   & 2 & \begin{pmatrix}  8 &  4 \\  4 & 12 \end{pmatrix} & (8) & 2 & [ 1920, 240995 ] & \textbf{81b} \\
  & 4 & \begin{pmatrix}  4 &  0 \\  0 &  4 \end{pmatrix} & (40) & 2 & - & \textbf{81c} \\ \hline
\end{array}
\end{equation*}

A projective model of \textbf{81a} is given in \cite{mukai:symplectic}:
\begin{equation}
 x^4+y^4+z^4+w^4+12 xyzw=0 \quad\text{in}\quad \P^3.
\end{equation}
We have $\nsym{G}=M_{20} \rtimes \mu_2$, where $\mu_2$ is generated by
% comma
% $(x,y,z,w) \mapsto (y,x,z,w)$.
% colon
 $(x : y : z : w) \mapsto (y : x : z : w)$.

A projective model of \textbf{81b} is given in \cite{bonnaf-sarti:maximal}:
\begin{align}
q_1 &= x_1^2 - x_4^2 -\phi x_5^2 +  \phi x_6^2=0,\\
q_2 &= x_2^2 + \phi x_4^2 + x_5^2 -\phi x_6^2=0,\\
q_3 &= x_3^2 - \phi x_4^2 - \phi x_5^2  + x_6^2=0,
\end{align}
where $\phi = (1+\sqrt{5})/2$ is the golden ratio.
The group $\nsym{G}=\sym{G} \rtimes \mu_2$ is generated by
\[
\left(\begin{smallmatrix}
-1  &  0  &  0  &  0  &  0  &  0 \\
0  &  1  &  0  &  0  &  0  &  0\\
0  &  0  &  1  &  0  &  0  &  0\\
0  &  0  &  0  &  1  &  0  &  0\\
0  &  0  &  0  &  0  &  1  &  0\\
0  &  0  &  0  &  0  &  0  &  1
\end{smallmatrix}\right),
\left(\begin{smallmatrix}
i  &  0  &  0  &  0  &  0  &  0 \\
0  &  0  &  1  &  0  &  0  &  0\\
0  &  1  &  0  &  0  &  0  &  0\\
0  &  0  &  0  & -i  &  0  &  0\\
0  &  0  &  0  &  0  &  0  &  1\\
0  &  0  &  0  &  0  &  1  &  0
\end{smallmatrix}\right),
\left(\begin{smallmatrix}
0 & 1 & 0 & 0 & 0 & 0\\
0 & 0 & 0 & 1 & 0 & 0\\
0 & 0 & 0 & 0 & 1 & 0\\
1 & 0 & 0 & 0 & 0 & 0\\
0 & 0 & 0 & 0 & 0 & 1\\
0 & 0 & 1 & 0 & 0 & 0
\end{smallmatrix}\right).\]

In \textbf{81c}, the group $G$ has the maximal finite order $960$. Its existence is proven in \cite{kondo:maximal} by applying lattice theory. For equations of \textbf{81c}, see Section \ref{sect:maximal}. The full automorphism group over $\C$ is calculated in \cite{keum-kondo:autom}; see  \cite{shimada:elliptic-modular} for mixed characteristic.

\section{The group of maximal order}\label{sect:maximal}

In this section, we give a projective model of \textbf{81c}.
Let $Y$ be the surface in $\mathbb{P}^5$ defined by the following equations:
%is birational to a K3 surface:
\begin{equation} \label{eq:degree8}
\begin{array}{cccccccc}
 f_1= & x_1^2 & +x_2^2 & +x_3^2 & -x_4^2 & & & =0, \\
 f_2= & & & ~\,\; x_3^2 & +x_4^2 & +x_5^2 & -x_6^2  & =0, \\
 f_3= & x_1^2 & -x_2^2 & & & +x_5^2 & +x_6^2 & =0.
\end{array}
\end{equation}
The surface $Y$ has a linear action of $H_s=C_2^4 \rtimes \MF{A}_4$,
 which is generated by
%\begin{equation}
%\left(\begin{smallmatrix}
% i  &  0  &  0  &  0  &  0  &  0 \\
%0  & -i  &  0  &  0  &  0  &  0\\
%0  &  0  &  0  &  1  &  0  &  0\\
%0  &  0  &  1  &  0  &  0  &  0\\
%0  &  0  &  0  &  0  &  0  &  i\\
%0  &  0  &  0  &  0  & -i  &  0
%\end{smallmatrix}\right),
%\left(\begin{smallmatrix}
%-1 &  0  &  0  &  0  &  0  &  0 \\
%0  &  1  &  0  &  0  &  0  &  0\\
%0  &  0  & -1  &  0  &  0  &  0\\
%0  &  0  &  0  &  1  &  0  &  0\\
%0  &  0  &  0  &  0  &  1  &  0\\
%0  &  0  &  0  &  0  &  0  &  1
%\end{smallmatrix}\right),
%\left(\begin{smallmatrix}
%0 & 0 & 1 & 0 & 0 & 0\\
%0 & 0 & 0 & 1 & 0 & 0\\
%0 & 0 & 0 & 0 & 1 & 0\\
%0 & 0 & 0 & 0 & 0 & 1\\
%1 & 0 & 0 & 0 & 0 & 0\\
%0 & 1 & 0 & 0 & 0 & 0
%\end{smallmatrix}\right).
%\end{equation}
\begin{align}
% \begin{pmatrix} i & 0 \\ 0 & -i \end{pmatrix} \oplus
% \begin{pmatrix} 0 & 1 \\ 1 & 0 \end{pmatrix} \oplus
% \begin{pmatrix} 0 & i \\ -i & 0 \end{pmatrix}, \text{~i.e.~}
% \begin{pmatrix} -1 & 0 \\ 0 & 1 \end{pmatrix} \oplus
% \begin{pmatrix} -1 & 0 \\ 0 & 1 \end{pmatrix} \oplus
% \begin{pmatrix} 1 & 0 \\ 0 & 1 \end{pmatrix},\\
%
% (1)
% comma
% (x_1,\ldots,x_6) & \mapsto (x_1,x_2,i x_4,i x_3,x_6,x_5),\\
% colon
 (x_1 : \cdots : x_6) & \mapsto (x_1 : x_2 : i x_4 : i x_3 : x_6 : x_5), \\
% (2)
% comma
% (x_1,\ldots,x_6) & \mapsto (-x_1,x_2,-x_3,x_4,x_5,x_6),\\
% colon
 (x_1 : \cdots : x_6) & \mapsto (-x_1 : x_2 : -x_3 : x_4 : x_5 : x_6),\\
% (3)
% comma
% (x_1,\ldots,x_6) & \mapsto (x_3,x_4,x_5,x_6,x_1,x_2).
% colon
 (x_1 : \cdots : x_6) & \mapsto (x_3 : x_4 : x_5 : x_6 : x_1 : x_2).
\end{align}
Moreover, $Y$ has an automorphism $h$ of order $4$:
\begin{equation}
% \begin{pmatrix} 1 & \\ & i \end{pmatrix} \oplus
% \begin{pmatrix} & & 1 & \\ & & & i \\ 1 & & & \\ & i & & \end{pmatrix}:
% \quad x_1^2 \leftrightarrow -x_2^2, ~ x_3 \leftrightarrow x_5, ~
% x_4^2 \leftrightarrow -x_6^2.
% comma
% h \colon (x_1,\ldots,x_6) \mapsto (x_1,i x_2,x_5,i x_6,x_3,i x_4).
% colon
 h \colon (x_1 : \cdots : x_6) \mapsto (x_1 : i x_2 : x_5 : i x_6 : x_3 : i x_4).
\end{equation}
There are $16$ singular points $p_1,\ldots,p_{16}$ of $Y$,
 e.g.\ $(0:1:0:1:0:1) \in \mathbb{P}^5$.
They form one orbit under the $H_s$-action and each of them is of type $A_1$.
Let $\pi \colon X \rightarrow Y$ be the minimal resolution.
Then $X$ is a K3 surface.
The induced action of $H_s$ on $X$
 is symplectic and we have $h^* \omega_X=i \omega_X$.
Let $l' \in \NS_X$ denote the pull-back of the class of hyperplane section of $Y$.
Furthermore, let $d \in \NS_X$
 denote the sum of the classes of the $16$ exceptional curves of $\pi$.
Then we have
\begin{equation}
 l:=3l'-d, \quad l^2=9 \cdot l'^2+d^2=9 \cdot 8+16 \cdot (-2)=40,
\end{equation}
and
\begin{gather}
 \H^0(X,l) \cong \{ s \in R_3 %\H^0(\mathcal{O}_{\mathbb{P}^5}(3))
  \bigm| s(p_1)=\cdots=s(p_{16})=0 \} / I_3
  \subset R_3/I_3 \cong \H^0(X,3 l'),
%  / \{ x_i f_j \bigm| 1 \leq i \leq 6, ~ 1 \leq j \leq 3 \}_\mathbb{C}.
\end{gather}
where $R_3$ and $I_3$ are the homogeneous parts of degree $3$ of $R := \mathbb{C}[x_1,\ldots,x_6]$ and the defining ideal $I$ of $Y$, respectively.
(Hence $I_3$ is spanned by $x_i f_j$ for $1 \leq i \leq 6$ and $1 \leq j \leq 3$.)
We take the following basis of $\H^0(X,l) \cong \C^{22}$:
\begin{align*}
 (z_1,\ldots,z_{22})
  &=(x_1 x_2 x_3,x_1 x_2 x_4,x_3 x_4 x_5,x_3 x_4 x_6,x_1 x_5 x_6,x_2 x_5 x_6, \notag \\
  & \quad~~ x_1 x_2 x_5,x_1 x_2 x_6,x_1 x_3 x_4,x_2 x_3 x_4,x_3 x_5 x_6,x_4 x_5 x_6, \\
  & \quad~~ x_1^2 x_2,x_1 x_2^2,x_3^2 x_4,x_3 x_4^2,x_5^2 x_6,x_5 x_6^2,
  x_1 x_3 x_5,x_1 x_4 x_6,x_2 x_3 x_6,x_2 x_4 x_5). \notag
\end{align*}
(The Riemann--Roch theorem also implies $\dim \H^0(X,l)=l^2/2+2=22$.)
The complete linear system for $l$
 gives a smooth embedding of $X$ into $\mathbb{P}^{21}$
 with coordinates $z_1,\ldots,z_{22}$.
Moreover, the coordinates $z_1,\ldots,z_6$ define a non-normal model $\overline{X}$ of
 $X$ in $\mathbb{P}^5$.
By using {\sc Singular} \cite{singular-cas}, one can check the following:
the defining ideal of $\overline{X}$ is generated by $(\overline{g}^i)^* q$ for $0 \leq i \leq 4$, where
\begin{gather}
 q := (-z_1^2+z_2^2-z_3^2-z_4^2) z_5^2+(z_1^2-z_2^2-z_3^2-z_4^2) z_6^2
 +z_5^4-z_6^4, \\
 \overline{g} \colon (z_1,\ldots,z_{6}) \mapsto
  (-i z_{2},-z_{3},-z_{5},-i z_{1},-z_{4},z_{6}).
\end{gather}
The automorphism of $\overline{X}$ (of order $5$) induced by $\overline{g}$ is extended to an automorphism $g$ of $X$ embedded into $\mathbb{P}^{21}$, where $g$ is given by
\begin{align}
% comma
%  g \colon (z_1,\ldots,z_{22}) \mapsto & (-i z_{2},-z_{3},-z_{5},-i z_{1},-z_{4},z_{6}, \notag \\ & ~~ i z_{12},-z_{10},z_{9},-z_{7},-z_{8},i z_{11}, \notag \\ & ~~ -i z_{16},i z_{22},z_{13},-i z_{19},z_{18},i z_{21}, -i z_{20},-i z_{15},-z_{14},z_{17}).
% colon
 g \colon (z_1 : \cdots : z_{22}) \mapsto & (-i z_{2} : -z_{3} : -z_{5} : -i z_{1} : -z_{4} : z_{6} :  \notag \\ & ~~ i z_{12} : -z_{10} : z_{9} : -z_{7} : -z_{8} : i z_{11} :  \notag \\ & ~~ -i z_{16} : i z_{22} : z_{13} : -i z_{19} : z_{18} : i z_{21} :  -i z_{20} : -i z_{15} : -z_{14} : z_{17}).
\end{align}
By a direct computation, one can check that the following Cremona transformation acts on $Y$ and induces $g$:
\begin{equation}
% comma
% (x_1,\ldots,x_6) \mapsto (x_3 x_4 , -x_2 x_5 , x_1 x_2 , -i x_3 x_5 , -x_5 x_6 , -i x_2 x_3).
% colon
 (x_1 : \cdots : x_6) \mapsto (x_3 x_4  :  -x_2 x_5  :  x_1 x_2  :  -i x_3 x_5  :  -x_5 x_6  :  -i x_2 x_3).
\end{equation}

%[x1=x3*x4,x2=-x2*x5,x3=x1*x2,x4=-%i*x3*x5,x5=-x5*x6,x6=-%i*x2*x3]
The group $G_s$ of all symplectic automorphisms of $X$ with polarization $l$
 is generated by $H_s$ and $g$.
We have $G_s \cong M_{20}$.
The group $G \cong M_{20} \rtimes \mu_{4}$ of all automorphisms of $X$ with polarization $l$ is generated by $G_s$ and $h$.

\begin{rem}
The motivation for this construction is as follows.
Let $X$ be a K3 surface with an action of $G=M_{20} \rtimes \mu_4$
 as in \textbf{81c}.
We consider a maximal proper subgroup $\sym{H}$ of $\sym{G}=M_{20}$ isomorphic to $2^4.\MF{A}_4$.
From \cite[Table 10.3]{hashimoto:symplectic}, we get $\rank \Lambda^\sym{H}=4$ and the genus symbol of $\Lambda^\sym{H}$ is $2^{-2}_{\II},8^{-2}_2$.
(In \cite{hashimoto:symplectic}, $H_s$ is No.\ 75 and its structure is written as $4^2\MF{A}_4$.)
Consider the lattice $L$ of rank $4$ with basis $(b_1,\ldots,b_4)$ and Gram matrix
\begin{equation}
 \begin{pmatrix}
            4 & 0 & 0 & 0 \\
            0 & 4 & 0 & 0 \\
            0 & 0 & 8 & 0 \\
            0 & 0 & 0 & -8
  \end{pmatrix}.
\end{equation}
The lattice $\overline{L}$ generated by $L$ and $(b_1+b_2+b_3)/2$
 is isomorphic to $\Lambda^\sym{H}$.
Consider an isometry of $h$ defined by
 $(b_1,\ldots,b_4)\mapsto(b_2,-b_1,b_3,b_4)$,
 which extends to $\overline{L}$.
By a lattice-theoretic argument, it follows that
 the action of $\mu_4$ on $\Lambda^\sym{H}$ corresponds to $\langle h \rangle$ and there is an ample class $l'$ of degree $8$ giving rise to $b_3$.
This suggests that there is a complete intersection of type $(2,2,2)$ in $\P^5$ birational to $X$,
 which is nothing but $Y$ above.
The classes $l$ and $d$ correspond to $3b_3-2b_4$ and $2b_4$, respectively.
Indeed, calculating the orthogonal complement of $l'$ inside the N\'eron-Severi lattice one finds $16$ vectors (up to sign) of square $(-2)$. These give the $16$ singular points of type $A_1$ of $Y$ and their sum is $2b_4$.
\end{rem}

\bibliography{maximal-symplectic.bib}{}
\bibliographystyle{hplain.bst}

\par\noindent
 Fakult\"at f\"ur Mathematik und Informatik, Universit\"at des Saarlandes,\\
 Campus E2.4, 66123 Saarbr\"ucken, Germany
\par\noindent{\ttfamily brandhorst@math.uni-sb.de}
\\

\par\noindent
Graduate School of Mathematical Sciences, The University of Tokyo,\\
3-8-1 Komaba, Maguro-ku, Tokyo, 153-8914, Japan
\par\noindent{\ttfamily kenji.hashimoto.math@gmail.com}

\end{document}